\documentclass[12pt,a4paper]{amsart}
\usepackage[latin1]{inputenc}
\usepackage[english]{babel}
\usepackage{amsfonts,amsmath}
\usepackage{fancyhdr,graphicx, xcolor}
\usepackage{cancel}
	\usepackage[hidelinks]{hyperref} 

\usepackage{amsthm} 
	
	\usepackage[scanall,209mode]{psfrag}

	\def\R{\mathbb R}

	\def\N{\mathbb{N}}   
	\def\R{\mathbb{R}}   
	\def\C{\mathbb{C}}   
	\def\DD{\mathcal{D}}
	\def\vjk{v^{j\to k}(t)}
	\DeclareMathOperator{\Diag}{Diag}

	\DeclareMathOperator{\conv}{conv}
	\DeclareMathOperator{\tr}{tr}

	\def\WASE{W(\mathbf{A_{S,E}})}
	\def\ASE{\mathbf{A_{S,E}}}
	
	\def\restrict#1{\raise-.5ex\hbox{\ensuremath{\big|}}_{#1}} 

	\newtheorem{theorem}{Theorem}

	\newtheorem{corollary}{Corollary}
	
	\newtheorem{definition}{Definition}

	\newtheorem{lemma}{Lemma}

	\newtheorem{proposition}{Proposition}
	\newtheorem{remark}{Remark}

	\numberwithin{equation}{section}
	
\begin{document} 
		
		\title[Minimal compact operators and joint numerical range]{Minimal self-adjoint compact operators, moment of a subspace and joint numerical range}
		
		\author{Tamara Bottazzi $^{1,2}$ and Alejandro Varela$^{3,4}$}

		\address{$^1$ Universidad Nacional de R\'io Negro. Centro Interdisciplinario de Telecomunicaciones, Electrónica, Computación y Ciencia Aplicada, Sede Andina (8400) S.C. de Bariloche, Argentina.}
		\address{$^2$ Consejo Nacional de Investigaciones Cient\'ificas y T\'ecnicas, (1425) Buenos Aires, Argentina.}
		\address{$^3$Instituto Argentino de Matem\'atica ``Alberto P. Calder\'on", Saavedra 15 3er. piso, (C1083ACA) Buenos Aires, Argentina}
		\address{$^4$Instituto de Ciencias, Universidad Nacional de Gral. Sarmiento, J. M. Gutierrez 1150, (B1613GSX) Los Polvorines, Argentina}
		\email{tbottazzi@unrn.edu.ar, avarela@campus.ungs.edu.ar }
		\thanks{Partially supported by Grants CONICET (PIP 0525), ANPCyT (PICT 2015-1505 and 2017-0019) and UNRN (PI 40-B-906)}
		\maketitle

		\begin{abstract}
			We define the (convex) joint numerical range for an infinite family of compact operators in a Hilbert space $H$. We use this set to determine whether a self-adjoint compact operator $A$ with $\pm\|A\|$ in its spectrum is minimal respect to the set of diagonals in a fixed basis $E$ of $H$ in the operator norm, that is $\|A\|\leq \|A+D\|$, for all diagonal $D$. We also describe the moment set $m_S={\conv}\left\{|v|^2: v \in S \text{ and } \|v\|=1 \right\}$ of a subspace $S\subset H$ in terms of joint numerical ranges and obtain equivalences between the intersection of moments of two subspaces and of its two related joint numerical ranges. Moreover, we relate the condition of minimality of $A$ or the intersection of the moments of the eigenspaces of $\pm\|A\|$ to the intersection of the joint numerical ranges of two finite families of certain finite hermitian matrices. We also study geometric properties of the set $m_S$ such as extremal curves related with the basis $E$.  All these conditions are directly related with the description of minimal self-adjoint compact operators.
		\end{abstract}

\section{Introduction and preliminaries}

Given a Hilbert space $H$, we call $A\in B(H)$ a minimal operator if $\|A\|\leq\|A+D\|$, for all $D$ diagonal in a fixed orthonormal basis $E=\{e_i\}_{i\in I}$ of $H$ and $\|\cdot\|$ the operator norm. Note that in the case $A$ is a compact operator we can suppose that $H$ is separable since there is only a numerable set $\{e_{i_k}\}_{k\in \N}$ such that $A(e_{i_k})\neq 0$. As mentioned in the literature, these operators allow the concrete description of geodesics in homogeneous spaces obtained as orbits of unitaries under a natural Finsler metric (see  \cite{dmr1}).

 In \cite{bottazzi-varela-minimal-compacts} we studied minimal self-adjoint compact operators where it was stated that in general neither existence nor uniqueness of compact minimizing diagonals was granted. Some of these results were recently generalized to more general subalgebras of $K(H)$ and to C$^*$-algebras in \cite{Zhang-Jiang-2022, Zhang-Jiang-2023}.
%


Given a subspace $S\subset H$ we call the moment set of $S$ to
$$
m_S=\Diag\{Y\in \mathcal{B}_1(H): Y\geq 0, P_SY=Y, \tr(Y)=1\}
$$ 
for $\mathcal{B}_1(H)$ the ideal of trace class operators. Equivalently $m_S= \text{convex hull}\left\{|v|^2: v \in S \text{ and } \|v\|=1 \right\}$, where $|v|^2=(|v_1|^2,|v_2|^2,\dots)$ for $v=(v_1,v_2,\dots)$ the coordinates of $v$ in the $E$ basis  (see Proposition \ref{prop: equivalencias de momento}). These sets are fundamental in the detection and parametrization of minimal self-adjoint compact operators. More specifically, for $S\perp V$ finite dimensional subspaces, $R$ a self-adjoint operator such that $\text{ran}(R)\perp S, V$ and $\|R\|\leq 1$, follows that $m_S\cap m_V\neq \{0\}$ if and only if $A= P_S-P_V+R$ is a minimal compact self-adjoint operator. And every compact minimal self-adjoint operator $A$ of norm 1 can be described in this way (see \cite{bottazzi-varela-minimal-compacts}).

In finite dimensions, recent results obtained in \cite{kloboukvarela-mom-jnr} describe properties of minimal $n\times n$ hermitian matrices $M\in M_n^h(\C)$ relating them to certain joint numerical ranges. These are the (convex) joint numerical ranges of $\{E_iP_+E_i\}_{i=1}^n$ and $\{E_iP_-E_i\}_{i=1}^n$ where $P_+$ and $P_-$ are the orthogonal projectors onto the eigenspaces of $\|M\|$ and $-\|M\|$ and $E_i=e_i\otimes e_i$ are the rank one projections onto span$\{e_i\}$, for $e_i\in E$. Our interest in these relation is because there are many properties already studied to describe joint numerical ranges that can be applied to determine when the condition $m_S\cap m_V\neq \emptyset$ holds (see Theorem \ref{prop equivs mS cap mV no vacio}).

In the present work we will generalize the main results of \cite{kloboukvarela-mom-jnr} to the context of compact operators. For this purpose we need to define the (convex) joint numerical range of an infinite family of compact operators (see Definition \ref{def JNR}). 
In Section \ref{jnr en sucesiones} we present the joint numerical range $W(\mathbf{A})$, joint numerical radius  $w(\mathbf{A})$ and moment set $m_S$ in this context. We study some of their general properties and the close relation of $W(\mathbf{A})$, for the particular family $A=\{P_SE_iP_S\}_{i=1}^\infty$, with $m_S$ (see Proposition \ref{coro: W(ASE) en funcion de mS}). 

Section \ref{sec curves} includes the description of some extremal points and curves of $m_S$ related with the principal angles of $S$ with the canonical subspaces generated by the elements $e_i$ of the fixed basis $E$, $i\in \N$.

In Section \ref{secc mS y traza cero} we relate the real space of hermitian operators defined on a subspace $S\subset H$ of $\dim S=r<\infty$ with the  real subspace $M^h_r(\C)$ of hermitian $r\times r$ matrices. This is done using an explicit (real) isometric isomorphism $U$ between $M_r^h(\C)$ and $\mathcal{B}_S^h=P_S B^h(H)P_S$ (see Proposition \ref{props unitario U}) constructed by a Gell-Mann generalized basis.
Moreover, $U$ and its inverse preserve joint numerical ranges. These results allow the description of the condition of non empty intersection of moments of two finite dimensional orthogonal subspaces $S$, $V\subset H$ using a finite family of $\dim S\times \dim S$ and $\dim V\times \dim V$ matrices. This can be done verifying the equivalent condition of not null intersection of certain joint numerical ranges. These conditions allow the construction of all minimal compact operators.

As a consequence we can prove our main result that relates moments of orthogonal finite dimensional subspaces of $H$, joint numerical ranges, supports of a pair of  subspaces and minimal hermitian compact operators (see Theorem \ref{prop equivs mS cap mV no vacio} in Section \ref{secc minimalidad con finitas mat}).

\section{Joint numerical range for a sequence of operators in $\mathcal{K}(H)$}\label{jnr en sucesiones}
Let $\mathcal{K}(H)$ be the ideal of compact operators in a separable Hilbert space $H$, where we denote with $\langle\, ,\, \rangle$ its inner product and $\|\, \|$ the induced norm. In the algebra of bounded operators $B(H)$ we will also use $\|\, \|$ to indicate the operator norm.

The dual space of the compact operators  $\left(\mathcal{K}(H)\right)^*=\mathcal{B}_1(H)$ is the ideal of trace-class operators  $T\in\mathcal{K}(H)$ such that $\tr |T|<\infty$ (where $\tr$ denotes the usual trace).
   
\begin{definition}\label{def JNR}
Consider a sequence $\mathbf{A}=\{A_j\}_{j=1}^\infty\in \mathcal{K}(H)^\mathbb{N}$ of self-adjoint compact operators $A_j$ with bounded norm ($\|A_j\|\leq c$, for all $j$).
We define the joint numerical range of $\mathbf{A}$ by
	\begin{equation}
	\label{eq def JNR}
 W\left(\mathbf{A}\right)=\left\{\{\tr \left(\rho A_j\right)\}_{j=1}^\infty: \rho\in 
\mathcal{B}_1(H)	\wedge \tr(\rho)=1 \wedge \rho\geq 0\right\}.
\end{equation}
\end{definition}	
 Note that $|\tr(\rho A_j)| \leq \|A_j\|\, \tr(\rho)\leq c$ which implies $\{\tr \left(\rho A_j\right)\}_{j=1}^\infty\in \ell^\infty(\mathbb{R})$ and therefore $ W\left(\mathbf{A}\right)\subset\ell^\infty(\mathbb{R})$. By the linearity of the trace and the convexity of the set $\mathcal{D}=\{\rho\in 
 	\mathcal{B}_1(H)^h:\ \tr(\rho)=1 \wedge \rho\geq 0\}$ it is evident that $W(\mathbf{A})$ is a convex set.
 	
\begin{definition}\label{radius}
We also define for the family $\mathbf{A}=\{A_j\}_{j=1}^\infty\in \mathcal{K}(H)^\mathbb{N}$ with $\|A_j\|\leq c$, for all $j$, the $p$-joint numerical radius as
\begin{equation}\label{wp}
w_p(\mathbf{A})=\sup\left\lbrace \left( \sum_{j\in \N}|\tr(\rho A_j)|^p\right)^{1/p}:\ \rho\in \mathcal{B}_1(H)	\wedge \tr(\rho)=1 \wedge \rho\geq 0  \right\rbrace, \text{ for every } p\in[1,+\infty).
\end{equation}
\end{definition}
Clearly, $w_p(\mathbf{A})$ may be $+\infty$ depending on the family $\mathbf{A}$.

Given a subspace $S$ of $H$ we will consider the set of its density operators 
\begin{equation}
	\label{def DsubS}
\mathcal{D}_S=
\left\{Y\in \mathcal{B}_1(H):P_S Y=Y\geq 0\ ,\text{tr}(Y)=1\right\}
\end{equation}
(note that $P_S Y=YP_S=P_S YP_S$ for $Y\in\DD_S$). If $\dim S<\infty$ the affine hull of $\DD_S$ is also finite dimensional.

The next result is a generalization of Lemma 6.1 in \cite{kloboukvarela-mom-jnr}.
\begin{lemma}\label{lem rho} Let $S$ be a subspace of $H$ and $\mathcal{D}_S$ as in \eqref{def DsubS}, then
	$$
	\mathcal{D}_S=\{\rho\in \mathcal{B}_1(H):\ \rho\geq 0,\ \tr(\rho)=\tr(P_S\rho P_S)\}.
	$$
\end{lemma}
\begin{proof}
	Let $Y\in \mathcal{B}_1(H)$ be such that $P_S Y=Y\geq 0$. Then, $\tr(P_SYP_S)=\tr(YP_S)=\tr(Y)=1$, which implies that 
	$Y\in \{\rho\in \mathcal{B}_1(H) :\ \rho\geq 0,\ \tr(\rho)=\tr(P_S\rho P_S)\}$. The reverse inclusion follows the same ideas in \cite[Lemma 6.1]{kloboukvarela-mom-jnr}.
\end{proof}

Now, motivated by the finite dimensional case of the \textsl{moment of a subspace} $S$ studied in \cite{kloboukvarela-mom-jnr} and \cite{soportes}, we define 
\begin{equation}
\begin{split}
\label{def momento}
m_S&=\Diag(\mathcal{D}_S)\\
&=\{\text{Diag}(Y): Y\in \mathcal{D}_S\}\ \subset  \left\{x\in \ell^1(\R): x_j\geq 0 \text{ and } \sum_{j=1}^\infty x_j=1\right\}
\end{split}
\end{equation}
where $\Diag(K)$ indicates the diagonal compact operator with the same diagonal than $K\in\mathcal{K}(H)$ with respect to a standard (fixed) basis $E=\{e_i\}_{i=1}^\infty$ of $H$.  We will also identify the diagonal matrices of $\text{Diag}(\mathcal{D}_S)$ with the corresponding sequences in $\ell^1(\R)$.

\begin{remark}
	\label{rem: sobre props mS}
In infinite dimensions the set $m_S$ was used in the proof of $(3)\Rightarrow (2)$ of \cite[Theorem 7]{bottazzi-varela-minimal-compacts}, where $S$ is the eigenspace of $\|A\|$ or $-\|A\|$ for $A$ a minimal self-adjoint compact operator (that is $\|A\|=\text{dist}(A,\text{Diag}\left( \mathcal{K}(H)\right)$).  
In this case dim$\left(\text{Ran}(P_S)\right)<\infty$, and then every $Y\in \mathcal{D}_S$ can be considered a self-adjoint operator between finite fixed dimensional spaces. Then, all norms restricted to those spaces are equivalent and $m_S=\text{Diag}(\mathcal{D}_S)$ is a compact and convex set for every norm.

Moreover, if $\{-\|A\|,\|A\|\}\subset \sigma(A)$, the non empty intersection between the corresponding moments related to the eigenspaces of $\|A\|$ and $-\|A\|$ implies that such a compact hermitian operator $A$ is minimal (see \cite[Corollary 10]{bottazzi-varela-minimal-compacts} and Proposition \ref{prop: equivalencias de momento}). 

\end{remark}

For $E=\{e_j\}_{j=1}^{\infty}$ we will denote with  $e_j\otimes e_j=E_j$, the rank-one orthogonal projections onto the subspaces generated by $e_j\in E$, for all $j\in\mathbb{N}$.
We will be particularly interested in the study of $W(\mathbf{A})$ in the case of 
$\mathbf{A}=\mathbf{A_{S,E}}=\{  P_SE_j P_S \}_{j=1}^\infty$ and $S$ a finite dimensional subspace of $H$ 
\begin{equation}\label{def W AsubS,E}
W\left(\mathbf{A_{S,E}}\right)=\left\{\{\tr \left(P_SE_j P_S\rho\right)\}_{j=1}^\infty: \rho\in \mathcal{B}_1(H) ,\ \rho\geq 0 \text{ and } \tr(\rho)=1\right\}.
\end{equation}
Observe that in this context
$$
\tr(P_SE_j P_S \rho)=\tr(E_j P_S \rho P_S E_j)=\langle P_S \rho P_S e_j,e_j\rangle=\left(P_S \rho P_S\right)_{j j}
$$
is the $j,j$ diagonal $E$-coordinate of the positive semi-definite trace-class operator $P_S \rho P_S$. Therefore
\begin{equation}\label{eq: serie de elem de W menor o igual que 1}
\sum_{j=1}^{\infty}\tr(E_j P_S\rho P_S)=\sum_{j=1}^{\infty}(P_S\rho P_S)_{j,j}=\tr(P_S\rho P_S)\leq \|P_S\| \tr(\rho)= 1
\end{equation}
which proves, in this case, that the sequences $\left\{\tr \left(P_SE_j P_S \rho\right)\right\}_{j=1}^\infty\in \ell^1\left(\mathbb{R}\right)\cap \mathbb{R}_{\geq 0}^\mathbb{N}$ and hence
\begin{equation}\label{eq W(AsubSE) esta en l1}
W\left(\mathbf{A_{S,E}}\right)\subset \ell^1\left(\R\right)\cap \mathbb{R}_{\geq 0}^\mathbb{N}.
\end{equation}

\begin{remark} \label{wp prop}
For the family $\mathbf{A_{S,E}}$, 
\begin{enumerate}
	\item the $p$-joint numerical radius $$w_p\left(\mathbf{A_{S,E}}\right)=\sup\left\lbrace \left( \sum_{j\in \N}\left( \tr(P_SE_jP_S\rho )\right) ^p\right)^{1/p}:\ \rho\in \mathcal{B}_1(H)	\wedge \tr(\rho)=1 \wedge \rho\geq 0  \right\rbrace$$
	is finite for every $p\in[1,\infty)$. This is a consequence of \eqref{eq W(AsubSE) esta en l1}, since every sequence $\{\tr(P_SE_jP_S\rho )\}_{j\in \N}\in \ell^1(\R)$.
	\item Moreover, $w_p\left(\mathbf{A_{S,E}}\right)\leq 1$ for every $p\in[1,\infty)$, since
	\begin{eqnarray*}
	\sum_{j\in \N}\left( \tr(P_SE_jP_S\rho )\right) ^p&=&\sum_{j\in \N}\left( P_S\rho P_S\right)_{jj} ^p=\|\text{Diag}(P_S\rho P_S)\|_p^p\\
	&\leq&\|P_S\rho P_S\|_p^p\leq\|P_S\|^p\|\rho P_S\|_p^p \leq\|P_S\|^{2p}\|\rho\|_p^p\leq \|\rho\|_1^p=1,
	\end{eqnarray*}
where the first inequality is due to the pinching property for Schatten $p$-norms (Theorem 1.19 in \cite{BarrySimon}).
	\item By \eqref{eq: serie de elem de W menor o igual que 1} and Lemma \ref{lem rho}, it can be deduced that
	$$w_1\left(\mathbf{A_{S,E}}\right)= 1.$$
\end{enumerate}
\end{remark}

Note that \eqref{eq: serie de elem de W menor o igual que 1}, \eqref{eq W(AsubSE) esta en l1} and Remark \ref{wp prop} hold for $S$ with $\dim(S)=\infty$.

The next result is a generalization from the finite dimensional case studied in Lemma 6.2 and Theorem 6.3 of \cite{kloboukvarela-mom-jnr}.

\begin{proposition}
	\label{prop: equivalencias de momento}
	The following are equivalent definitions of $m_S$, the moment of $S$ with $\dim S=r$, $r<\infty$, related to a basis $E=\{e_i\}_{i=1}^\infty$ of $H$. Note the identification made between diagonal operators and sequences.
	\begin{enumerate}
		\item[a) ] \label{defequiv mS 1 en prop} $m_S=\text{Diag}(\mathcal{D}_S)$.
		\item[b) ] \label{defequiv mS 2 en prop}
		$
		m_S= {\conv}\left\{|v|^2: v \in S \text{ and } \|v\|=1 \right\}.
		$
		\item[c) ] \label{defequiv mS 3 en prop} $m_S= \bigcup\limits_{\{s^i \}_{i=1}^r  \text{o.n. set in } S} \  {\conv}  \{|s^i|^2 \}_{i=1}^r.
		$
		\item[d) ] \label{eq relac momento y trEiY} 
		$m_S=\{\left(\tr(E_1 Y),\dots,\tr(E_n Y),\dots\right)\in\ell^1(\R): Y \in \DD_S \}$.
		\item[e) ] \label{eq relac momento y JNR} 
		$m_S=W(P_S E_1 P_S,\dots,P_S E_n P_S,\dots)\cap \left\{x\in \ell^1(\R):x_i\geq 0 \text{ and } \sum_{i=1}^\infty x_i=1\right\}$, 
		where $P_S$ is the orthogonal projection onto $S$, and $W$ is the joint numerical range from Definition 
		\ref{def JNR}.
	\end{enumerate}
\end{proposition}

\begin{proof}
Statement a) is Definition \eqref{def momento}.
Next we will consider some inclusions regarding the sets described in a), b) and c) to prove the equalities stated in those items. First observe that if $s\in S$ with $\|s\|=1$ then $Y=s\otimes s\in \DD_S$ because $\tr(s\otimes s)=\sum_{i=1}^\infty |s_i|^2=1$, $s\otimes s\geq 0$ and $P_S(s\otimes s)=s\otimes s$. Hence, since $\Diag(s\otimes s)=|s|^2$ and $m_S$ is convex, follows that ${\conv}\left\{|v|^2: v \in S \text{ and } \|v\|=1 \right\}\subset m_S=\Diag(\DD_S)$.
\\
Now if $\{s^i \}_{i=1}^r$  is an orthonormal set in $S$ then it is apparent that 
$$
{\conv}\{|s^i|^2 \}_{i=1}^r\subset{\conv}\left\{|v|^2: v \in S \text{ and } \|v\|=1 \right\}.
$$
This implies $\bigcup\limits_{\{s^i \}_{i=1}^r  \text{o.n. set in } S} \  {\conv}  \{|s^i|^2 \}_{i=1}^r
\subset{\conv}\left\{|v|^2: v \in S \text{ and } \|v\|=1 \right\}.
$
\\
Now take $Y\in m_S=\Diag(\DD_S)$. There exist an orthonormal basis $\{y_i\}_{i=1}^r$ of $S$ such that $Y=\sum_{i=1}^r \lambda_i (y_i\otimes y_i)$ with $\lambda_i\geq 0$ and $\sum_{i=1}^r \lambda_i=1$. Then $\Diag(Y)=\sum_{i=1}^r \lambda_i \Diag (y_i\otimes y_i)\simeq \sum_{i=1}^r \lambda_i |y_i|^2$ which is a convex combination of $\{|y_i|^2\}_{i=1}^r$ for the orthonormal set $\{y_i\}_{i=1}^r\subset S$. Then $\Diag(Y)\in \bigcup\limits_{\{s^i \}_{i=1}^r  \text{o.n. set in } S} \  {\conv}  \{|s^i|^2 \}_{i=1}^r$.
This proves that the sets described in the first three items are the same (using the identification of sequences with diagonal matrices in some cases).

	Now to prove statement d), take any $x=\text{Diag}(Y)\in m_S$ with $Y\in \mathcal{D}_S$ and $Y_{j,j}=\tr(E_jP_SYP_SE_j)=\tr(E_jYE_j)=\tr(E_jY)$ for every $j\in\mathbb{N}$. 

	In order to prove e) consider that using d) every $x\in m_S$ can be written as $x=\{\tr \left( E_j YE_j\right)\}_{j=1}^\infty\in \ell^1\left( \R \right)$, with $Y\in \mathcal{D}_S$. Then $x\in W\left(\mathbf{A_{S,E}}\right)$ and 
	$$
	\sum_{j=1}^{\infty}x_j=\sum_{j=1}^{\infty}\tr \left( E_j YE_j\right)=\tr(Y)=1.
	$$
	On the other hand, take $x\in W\left(\mathbf{A_{S,E}}\right)\cap \left\{x\in \ell^1(\R):x_i\geq 0 \text{ and } \sum_{i=1}^\infty x_i=1\right\}$, then there exists $\rho_0\in \mathcal{B}_1(H)$, $\rho_0\geq 0$, $\tr(\rho_0)=1$ such that
	$$x=\{\tr \left(P_S E_j P_S \rho_0\right)\}_{j=1}^\infty,\ \sum_{j=1}^{\infty}\tr \left( P_SE_j P_S \rho_0\right)=1.$$
	Therefore, $Y=P_S\rho_0P_S$ fulfills that $Y\geq 0$ and 
	$$1=\sum_{j=1}^{\infty}\tr \left(P_S E_j P_S \rho_0\right)=\sum_{j=1}^{\infty}\tr \left(E_j P_S \rho_0P_SE_j\right)=\sum_{j=1}^{\infty}\left( P_S \rho_0P_S\right)_{jj}=\sum_{j=1}^{\infty}Y_{jj}=\sum_{j=1}^{\infty}\left( P_S YP_S\right)_{jj}.$$
	Then, $Y\in \mathcal{D}_S$ and $x\in m_S$ by Lemma \ref{lem rho}.	
\end{proof}

In the same context, we can define the classic joint numerical range
\begin{definition}
	Consider a sequence $\mathbf{A}=\{A_j\}_{j=1}^\infty\in \mathcal{K}(H)^\mathbb{N}$ of self-adjoint hermitian compact operators $A_j$ with bounded norm ($\|A_j\|\leq c$, for all $j$).
	We define the classic joint numerical range of $\mathbf{A}$ by
	\begin{equation}
	\label{eq def JNRClass}
	W_{class}\left(\mathbf{A}\right)=\left\{\{\left\langle A_jx,x\right\rangle \}_{j=1}^\infty: x\in H	,\ \|x\|=1\right\}.
	\end{equation}
\end{definition}
Note that $|\left\langle A_jx,x\right\rangle | \leq \|A_jx\|\leq \|A_j\|\leq c$ which implies $\{\left\langle A_j x,x\right\rangle\}_{j=1}^{\infty}\in \ell^\infty(\mathbb{R})$ and therefore $ W_{class}\left(\mathbf{A}\right)\subset\ell^\infty(\mathbb{R})$.

In the particular case when $\mathbf{A}= \mathbf{A_{S,E}}$ then $W_{class}\left(\mathbf{A}_{S,E}\right)\subset\ell^1(\mathbb{R})$. This follows because $\rho_x=x\otimes x\in \mathcal{D} $ and $\tr (P_S E_i P_S\rho_x)=|(P_Sx)_{i,i}|^2=|\left\langle P_Sx,e_j\right\rangle|^2$, which implies that $\sum_{i=1}^\infty |(P_Sx)_{i,i}|^2= \|P_Sx\|^2\leq 1$ and
$$W_{class}\left(\mathbf{A_{S,E}}\right)=\left\{\{|\left\langle P_Sx,e_j\right\rangle|^2\}_{j=1}^\infty: x\in H	,\ \|x\|=1\right\}=\left\{|P_Sx|^2: x\in H	,\ \|x\|=1\right\}.$$

\begin{definition}\label{classradius}
	By extension, we define for $\mathbf{A}=\{A_j\}_{j=1}^\infty\in \mathcal{K}(H)^\mathbb{N}$ with $\|A_j\|\leq c$, for all $j$, the classic $p$-joint numerical radius as
	\begin{equation}\label{wpclass}
	w_{class,p}(\mathbf{A})=\sup\left\lbrace \left( \sum_{j\in \N}|\left\langle A_jx,x \right\rangle |^p\right)^{1/p}: x\in H, \|x\|=1 \right\rbrace,\text{ for } 1\leq p\le \infty
	\end{equation}
\end{definition}
And, as it occurs with $w_p(\mathbf{A})$, $w_{class,p}(\mathbf{A})$ may be $\infty$, and it depends on the family $\mathbf{A}$. Indeed, observe that if we consider a fixed unitary $x\in H$ and define $\bar{x}$ such that $\bar{x}_j=\left\langle A_jx,x \right\rangle$, for $j\in\N$, then
$$\|\bar{x}\|_p=\left( \sum_{j\in \N}|\left\langle A_jx,x \right\rangle |^p\right)^{1/p}\leq \|\bar{x}\|_1,$$
for every $p\geq 1$ since $\bar{x}\in W_{class}(\mathbf{A_{S,E}})\subset \ell^1(\mathbb{R})$.	
Therefore, $w_{class,p}(\mathbf{A_{S,E}})$ is a finite number for every $p\geq 1$.

\begin{remark}
Observe that $W_{class}$ is not a convex set even for a finite family $\mathbf{A}$ of cardinal greater than one (there are several examples in the literature, such as in 	\cite{gutkin-jonckheere-karow}, \cite{li-poon} and \cite{Mu-Tom}).
\end{remark}

\begin{proposition}
	If $\dim S<\infty$ and $W_\text{class}\left(
	\ASE\right)$ is convex then 
	\begin{equation}\label{eq WclassASE convexo igual WASE}
	W_\text{class}\left(\ASE \right)=W\left(\ASE \right).
	\end{equation}
\end{proposition}
\begin{proof}
	Recall that
	$W_\text{class}\left(\ASE\right)=\{(\tr (P_S E_1 P_S(x\otimes x ),
	\dots,P_S E_n P_S(x\otimes x ),\dots):x\in H, \|x\|=1\}=
	\{( \langle P_S E_1 P_S x,x \rangle ,
	\dots,\langle P_S E_n P_S x,x \rangle, 
	\dots):x\in H, \|x\|=1\}$.
	Then since 
	
	$|s|^2= (\langle P_S E_1 P_S s,s \rangle ,	\dots,\langle P_S E_n P_S s,s \rangle, \dots)$
	holds that  
	$\{|s|^2 : s\in S, \|s\|=1\} \subset W_\text{class}(\ASE)$. 
	Now item b) of Proposition \ref{prop: equivalencias de momento} and the assumed convexity of $W_\text{class}\left(\ASE\right)$ imply that
	$$
	m_S=\conv\{|s|^2 : s\in S, \|s\|=1\}\subset W_\text{class}\left(\ASE\right)
	$$
	The same arguments used to prove \eqref{eq: 0 pertenece a WASE} give that $(0,\dots,0,\dots)\in W_\text{class}\left(\ASE\right)$ and hence the convexity of $W_\text{class}\left(\ASE\right)$ imply that
	$$
	\{t\, x: 0\leq t\leq 1 \text{ and } x\in m_S\}	\subset W_\text{class}\left(\ASE\right)
	$$
	Corollary \ref{coro: W(ASE) en funcion de mS} and the fact that the inclusion 
	$W_\text{class}\left(\ASE \right)
	\subset W\left(\ASE \right)$
	always holds proves equality \eqref{eq WclassASE convexo igual WASE}.
\end{proof}

\begin{proposition}\label{coro: W(ASE) en funcion de mS}
			Following the notations of $\DD_S$ from \eqref{def DsubS}, $W$ of \eqref{eq def JNR} from Definition \ref{def JNR} and $W\left(\mathbf{A_{S,E}}\right)$ from \eqref{def W AsubS,E}, the following equality holds
		\begin{equation}
			W\left(\mathbf{A_{S,E}}\right)
			=
			\{t\, x: 0\leq t\leq 1 \text{ and } x\in m_S\}	=\bigcup_{t\in [0,1]}\left\lbrace t\left(\tr(\mu P_SE_1P_S), \tr(\mu P_SE_2P_S),...  \right): \mu\in \mathcal{D}_S \right\rbrace \label{conos}
		\end{equation}
		and hence
		$$\text{cone} \left(W\left(\mathbf{A_{S,E}}\right)\right)=
		\text{cone}\left( m_S\right).$$
\end{proposition}

\begin{proof}
	The first equality in \eqref{conos} can be proved in a similar way as done in \cite[Proposition 6.4]{kloboukvarela-mom-jnr} and the beginning of Section 7 of the same paper.
	\\
	Consider $\rho_x=x\otimes x$ with $x\in S^\perp$, $\|x\|=1$. Then 
	\begin{equation}
		\label{eq: 0 pertenece a WASE}
		\left(\tr(P_S E_1 P_S\rho_x),\dots,\tr(P_S E_n P_S\rho_x),\dots\right)=(0,\dots,0,\dots)\in W\left(\mathbf{A_{S,E}}\right).
	\end{equation}
	Next observe that item e) of Proposition \ref{prop: equivalencias de momento} implies $m_S\subset W\left(\mathbf{A_{S,E}}\right)$, and then \eqref{eq: 0 pertenece a WASE} and the convexity of $\WASE$ prove that $\{t\, x: 0\leq t\leq 1 \text{ and } x\in m_S\}\subset \WASE$.
	\\
	Now consider a non-zero $w=(\tr(P_SE_1P_S\rho),\dots,\tr (P_SE_nP_S\rho),\dots)\in W(\mathbf{A_{S,E}})$. Then $w=t\, x$ for $t=\tr(P_S\rho P_S)\leq 1$ (see Equation \eqref{eq: serie de elem de W menor o igual que 1}) and $x=\frac{1}{\tr(P_S\rho P_S)} w\in m_S$ since $\sum_{i=1}^\infty x_i=1$ (item e) of Proposition \ref{prop: equivalencias de momento}). Hence $w=t \, x\in \{t\, x: 0\leq t\leq 1 \text{ and } x\in m_S\}$ and the inclusion
	$$
	\WASE\subset \{t\, x: 0\leq t\leq 1 \text{ and } x\in m_S\}
	$$
	holds.
	
	For the second equality in \eqref{conos}, consider $\rho\in \mathcal{D}$ and $\left(\tr(\rho P_SE_1P_S), \tr(\rho P_SE_2P_S),...\right)\in W\left(\mathbf{A_{S,E}}\right)$. We separate in two different cases:
	\begin{itemize}
		\item If $\tr(P_S\rho P_S)\neq 0$, then, there exist $t\in(0,1]$ (for example $t=\tr(P_S\rho P_S)$) and $\mu \in \mathcal{D}_S$ such that $P_S\rho P_S=t \mu$ and
		$$\left(\tr(\rho P_SE_1P_S), \tr(\rho P_SE_2P_S),...\right)=\left(\tr(P_S\rho P_S E_1P_S), \tr( P_S\rho P_S E_2P_S),...\right)$$
		\begin{eqnarray*}
			&=&\tr(P_S\rho P_S)\left(\frac{1}{\tr(P_S\rho P_S)}\tr(P_S\rho P_S E_1P_S),\frac{1}{\tr(P_S\rho P_S)} \tr( P_S\rho P_S E_2P_S),...\right)\\
			&=&t \left(\frac{1}{\tr(P_S\rho P_S)}\tr(P_S\rho P_SE_1P_S),\frac{1}{\tr(P_S\rho P_S)} \tr( P_S\rho P_SE_2P_S),...\right)\\
			&=& t \left(\tr(\mu P_S E_1P_S), \tr( \mu P_S E_2P_S),...\right),
		\end{eqnarray*}
		with $t\in (0,1]$.
		\item If $\tr(P_S\rho P_S)= 0$ and since $P_S\rho P_S\geq  0$, then $P_S\rho P_S= 0$. Therefore, 
		\begin{eqnarray*}
			\left(\tr(\rho P_SE_1P_S), \tr(\rho P_SE_2P_S),...\right)&=&\left(\tr(P_S\rho P_S E_1P_S), \tr( P_S\rho P_S E_2P_S),...\right)\\
			&=&(0,0,...)\\
			&=& 0 \left(\tr(\mu P_S E_1P_S), \tr( \mu P_S E_2P_S),...\right),
		\end{eqnarray*}
	\end{itemize}
\end{proof}

\begin{remark}
		Analogously as in Proposition \ref{coro: W(ASE) en funcion de mS}, it can be proved that for any family $\mathbf{A_{S,T}}=\{P_ST_nP_S\}_{n\in \N}$, with $\{T_n\}\subset \mathcal{K}(H)^h$,
		$$W(\mathbf{A_{S,T}})=\bigcup_{t\in [0,1]}\left\lbrace t\left(\tr(\mu P_ST_1P_S), \tr(\mu P_ST_2P_S),... \right): \mu\in \mathcal{D}_S \right\rbrace$$
		holds.
\end{remark}

We obtain the next upper bound for the Hausdorff distance between two moments, 
equipped with $\|z\|_{\infty}=\sup_{i\in \N}|z_i|$, for $z\in\ell^1(\C)$.
\begin{lemma}\label{distH}
		Let $S$ and $V$ subspaces of $H$. Then, $\text{dist}_H(m_S,m_V)\leq 2.$ 
		Moreover, if $S\perp V$, then
		$$\text{dist}_H(m_S,m_V)\leq 1.$$
	\end{lemma}
	\begin{proof}
		Let $S,T$ subspaces of $H$, $x\in m_S$ and $y\in m_V$. Then, $x=\{\tr(E_iY)\}_{i\in\N}$ with $Y\in \mathcal{D}_S$, $y=\{\tr(E_iZ)\}_{i\in\N}$ with $Z\in \mathcal{D}_V$ and  
		$$\|x-y\|_{\infty}=\sup_{i\in \N}|\tr(E_iY)-\tr(E_iZ)|=\sup_{i\in \N}|\tr(E_iYE_i)-\tr(E_iZE_i)|=\sup_{i\in \N}|Y_{i,i}-Z_{i,i}|\leq 2,$$
		since $\|Y\|_1=\|Z\|_1=1$.
		
		In the case $S\perp V$, observe that $Y,Z\geq 0$ and $YZ=YP_SP_VZ=0$ (disjoint support). Then, by Proposition 3 in \cite{BCS}
		$$\|Y-tZ\|=\|Y+tZ\|=\max\{\|Y\|;\|Z\|\}\leq 1, \forall t\in \C.$$
		Therefore, 
		$$\text{dist}_H(m_S,m_V)=\max\left\lbrace \sup_{x\in m_S}d(x,m_V);\ \sup_{y\in m_V}d(y,m_S) \right\rbrace\leq 1$$
		if $S\perp V$ (for any $S$ and $T$, $\text{dist}_H(m_S,m_V)\leq2$).		
\end{proof}


The following lines are inspired in Remark 5 of \cite{kloboukvarela-mom-jnr}. Let $S$ be a finite dimensional subspace of $H$. The element of $m_S$ defined by
\begin{equation}
	\label{def centroide}
	c(m_{S})=\frac{1}{\dim S} \sum_{i=1}^{\dim S} |s^i|^2=\frac 1{\dim S}\text{Diag}(P_S)
\end{equation}
for any orthonormal basis $\{s^1,s^2,\dots,s^r\}$ of $S$ fulfills some interesting symmetric properties in the moment set $m_S$.

\medskip

	Let $\text{aff}\left(X\right)$ denote the affine hull of $X\subset B(H)$. 
	Since $\mathcal{D}_S$ can be characterized as a subset of $M_n^h(\mathbb{C})$, then  $\dim\big(\text{aff}\left(\mathcal{D}_S\right)\big)<\infty$ and $\dim\big(\text{aff}\left(\text{Diag}(\mathcal{D}_S)\right)\big)<\infty$.
	Hence the following result follows with almost the same proof of its finite dimensional counterpart in \cite[Proposition 3.4]{kloboukvarela-mom-jnr} 
	by an application of the Hahn-Banach hyperplane separation theorem.

%

\begin{proposition}\label{prop centroide es interior a mS en caps afin}
	Let $S\subset H$ be a subspace of $\dim(S)\geq 2$. Then $\dim\big(\text{aff}\left(\mathcal{D}_S\right)\big)<\infty$, 
	$\dim\left(\text{aff}(m_S)\right)<\infty$ 
	 and $c(m_{S,E})$ is an interior point of $m_S$ relative to the affine hull of $m_S$.
\end{proposition} 
\begin{proof} The finiteness of the dimensions of $\text{aff}\left(\mathcal{D}_S\right)$ and $\text{aff}(m_S)$ was discussed in the previous paragraph.
	
	Now suppose that $c=c(m_{S,E})$ is not an interior point relative to the affine hull $\text{aff}(m_S)$ of $m_S$ with  dim$(\text{aff}(m_S))=d$.	Then the compactness and convexity of $m_S$ (see Remark \ref{rem: sobre props mS}) imply that $c$ belongs to its boundary.	
	Now consider $\text{aff}(m_S)=c+T\subset \ell^1(\R)$ for a real subspace $T$, $\dim T=d$. With these assumptions there exists a functional $f:\ell^1(\R)\to \R$ such that $f(c)=k$ and $f(x)\leq k$, $\forall x\in m_S$. Let us suppose that there exists $v\in S$, $\|v\|=1$ such that $|v|^2\in m_S$ and $f(|v|^2)<k$.
	Now extend the vector $v=s^1$ to an orthonormal basis $\{s^i\}_{i=1}^r$ of $S$ (with $r=\dim S$). Then from the definition of $c$ in \eqref{def centroide}, it follows that $c=\frac1{r} \sum_{i=1}^{r} |s^i|^2$, and therefore, using the linearity of $f$
	\begin{equation}
		\begin{split}
			k&=f(c)=\frac1{r} \sum_{i=1}^{r} f\left(|s^i|^2\right) 
			\ \Rightarrow \\
			& \Rightarrow\ 
			k=f(c)=\frac1{r}  f\left(|s^1|^2\right)+\frac1{r}\sum_{i=2}^{r} f\left(|s^i|^2\right) <\frac k{r}+\frac1{r}\sum_{i=2}^{r} f\left(|s^i|^2\right)\leq \frac k{r}+\frac1{r}\sum_{i=2}^{r} k=k,
		\end{split}
	\end{equation}
	which is a contradiction.
	Using the characterization of $m_S$ from Proposition \ref{prop: equivalencias de momento} b) it must be $f(|x|^2)=k$ for every $x\in S$, with $\|x\|=1$. But this implies that $\text{aff}(m_S)$ has at least one dimension less than $d$. Then $c$ cannot be a boundary point of $m_S$ in $\text{aff}(m_S)$.
\end{proof}

\begin{remark}\label{rem aff hull mS}
	Note that the real affine hull of $m_S$ is $\text{aff}(m_S)=\Diag\left(\mathcal{B}_S^h\right)$ where $\mathcal{B}_S^h=\{X\in B(H): P_S X=XP_S \text{ and } X^*=X\}$.
\end{remark}

\begin{proposition}		\label{prop: varias props centroide}
			Let $S$ be a non-trivial finite dimensional subspace of $H$ with $\dim S=r$, and $E$ be a fixed basis of $H$ and $c(m_S)$ defined as in \eqref{def centroide}. Then $c(m_S)$ satisfies the following properties.
			\begin{enumerate}
				\item $c(m_S)\in m_S$. 
				\item  $c(m_S)$ coincides with the barycenter or centroid of the simplex generated by  $\{|w^1|^2,|w^2|^2,\dots,|w^r|^2\}\subset \R^\N_{\geq 0}$ obtained from any orthonormal basis $\{w^1, w^2,\dots,w^r\}$ of $S$.
				\label{prop cmSE es centroide de cualquier bon}
				\item Let $V$ another subspace of $H$ with $\dim V=k$, such that $S\perp V$. Then,
				\begin{equation}
					\label{eq ecuac centroide suma subesp perps}
					c\left(m_{S\perp V}\right)=\frac{1}{r+k}(r\, c(m_S)+k\, c(m_V)).
				\end{equation} This can be generalized to any number of mutually orthogonal subspaces. 
				\item Given a subspace $D\subset S$, with $\dim D = d < \dim S=r$, then $c(m_{S\ominus D})=c(m_{S\cap D^\perp})=\frac1{r-d}\left(r \, c(m_S)-d\, c(m_D)\right)$.
				\item  Let $S$ and $V$ be two subspaces of $\mathbb{C}^n$ with dimensions $r$ and $k$ respectively, and $D= S \cap V$ of dimension $d$ such that $\left(S\cap D^{\perp}\right)\perp \left(V\cap D^{\perp}\right)$ holds. Then $c(m_{S+V})=\frac{1}{r+k-d}(r\, c(m_S)+k\, c(m_V) - d\, c(m_D))$.
			\end{enumerate}

\end{proposition}	

The proof follows the same ideas of the corresponding ones in \cite[Proposition 3.5]{kloboukvarela-mom-jnr}.

\begin{remark}\label{rem comentario sobre centroide}
	Note the similarity of the equation \eqref{eq ecuac centroide suma subesp perps} with the one used to calculate the geometric centroid or barycenter of $m$ disjoint sets $A_j$ with $j=1,\dots,m$ using $c(\cup_{j=1}^m A_j)=\frac{\sum_{j=1}^m c(A_j) \mu(A_j)}{\sum_{j=1}^m\mu(A_j)}$, where $\mu$ is the corresponding measure.
\end{remark}

As it was done in \cite{soportes} in finite dimensions, we define analogously the notion of a pair subspaces of $H$ that form a support (see \cite[Theorem 3]{soportes} for some equivalent definitions of a support).

	\begin{definition}\label{soporte def}
		Let $S$ and $T$ subspaces of $H$ such that $\dim S=p$ and $\dim T=q$. We say that the pair $(S,T)$ forms a support if $m_S\cap m_T\neq \emptyset$, or equivalently, if there exists orthonormal sets $\{v^i\}_{i=1}^p\subset S$ and $\{w^j\}_{j=1}^q\subset T$ such that
		\begin{equation}\label{sop}
			\sum_{i=1}^{p}\alpha_i|v^i|^2=\sum_{j=1}^{q}\beta_j|w^j|^2,
		\end{equation}
		with $\alpha_i,\beta_j\geq 0$, and $\sum_{i=1}^{p}\alpha_i=\sum_{i=1}^{p}\beta_j=1$. \end{definition}
Observe that Definition \ref{soporte def} can be stated also for infinite dimensional subspaces $S$ and $T$ of $H$ if there exist finite collections of orthogonal sets $\{v^i\}_{i=1}^p\subset S$ and $\{w^j\}_{j=1}^q\subset T$ that fulfill \eqref{sop}.
\begin{remark}
	According to definition \label{soporte} and \cite[Corollary 10]{bottazzi-varela-minimal-compacts}, given $C\in \mathcal{K}(H)^h$
	with $\pm \|C\|\in \sigma(C)$, then $C$
	is a minimal operator if and only if the pair $(S_+,S_-)$ is a support,	where $S_+$ and $S_-$ are the corresponding eigenspaces of $\pm \|C\|$.
\end{remark}

\section{Principal vectors and curves of extremal points in $m_S$} \label{sec curves}
In this section we generalize the definition of principal (standard) vectors given in Definition 4.2 of \cite{kloboukvarela-mom-jnr} to obtain the description of curves of extreme points in the moment set $m_S$. We include results that are a natural generalization of the ones contained in Sections 4 and 5 of \cite{kloboukvarela-mom-jnr}.
\subsection{Principal standard vectors}
\begin{definition}\label{def gen S}
		We call a subspace $S\subset H$ a generic subspace with respect to the basis $E=\{e_j\}_{j=1}^{\infty}$ if there exists $x\in S$ such that $\left\langle x,e_j\right\rangle\neq 0$ for every $j\in \N$.
		This definition is equivalent to any of the statements
		\begin{itemize}
			\item $S$ is not included in the subspace $\text{span}\{e_j\}^{\perp}$ for $j\in \N$,
			\item $(P_S(e_j))_j=\left\langle P_Se_j,e_j \right\rangle=\left\langle P_Se_j,P_Se_j \right\rangle=\|P_Se_j\|^2\neq 0$ for all $j\in \N$.
		\end{itemize}
	Note that $S$ can be infinite dimensional in this definition. Also observe that if $S$ is not generic, we can work in another Hilbert space $\hat{H}\subset H$ where $S$ can be embedded isometrically and such that $S$ is generic in $\hat{H}$. Hence, in what follows we will suppose we are working with generic subspaces $S$ of $H$.
	
\end{definition}

\begin{definition}\label{def vec ppales}
	Given a generic subspace $S$ of $H$, we denote by
\begin{equation}\label{vector ppal}
v^j=\frac{P_Se_j}{\|P_Se_j\|}
\end{equation}
the unique principal (unitary) vectors related to the standard basis $E$ that satisfy $(v^j)_j=v^j_j=\left\langle v^j,e_j\right\rangle=\|P_Se_j\|>0 $ and minimize the angle between $S$ and $\text{span}\{e_j\}$, that is
$$\left\langle v^j,e_j\right\rangle=\max_{s\in S,\|s\|=1}|\left\langle s,e_j\right\rangle|=\|P_Se_j\|\leq 1$$
\end{definition}
The uniqueness can be proved observing that if there exists $w\in S$ such that $\|w\|=1$ and $\left\langle w,e_j\right\rangle=\left\langle v^j,e_j\right\rangle $, then
$$\|v^j-w\|^2=\|v^j\|^2+\|w\|^2-2\text{Re}\left(\left\langle v^j,w \right\rangle \right)=0,$$
since $\left\langle v^j,w \right\rangle=\dfrac{\left\langle e_j,w\right\rangle }{\|P_Se_j\|}=1$.

\begin{lemma}\label{propi ps y vectores ppales}
The orthogonal projection $P_S$ can be written matricially  and its infinite associated matrix related to the basis $E$ has the following properties:
\begin{enumerate}
	\item $(P_S)_{ij}=\left\langle P_Se_i,e_j\right\rangle =\|P_Se_i\|v_j^i$, for every $i,j\in \N$.
	\item $(P_S)_{jj}=\left\langle P_Se_j,e_j\right\rangle =\|P_Se_j\|^2=(v_j^j)^2$.
	\item Since $P_S=P_S^*$, $\|P_Se_i\|v_j^i=\|P_Se_j\|\overline{v_i^j}$ and
	$$\dfrac{\overline{v_i^j}}{v_j^i}=\left\lbrace \begin{array}{lll}
	\dfrac{\|P_Se_i\|}{\|P_Se_j\|}\neq 0.&\text{if}& v_i^j,v_j^i\neq 0\\
	0&\text{if}& v_i^j=v_j^i= 0
	\end{array} \right. 
	$$
	\item For each $i,j\in \N$, $v_i^j=\left\langle v^j,e_i\right\rangle=\left\langle v^j,P_Se_i \right\rangle=\|P_Se_i\|\left\langle v^j,v^i\right\rangle  $. Therefore, $v^j_j>0$ and
	$$0=v_i^j\Leftrightarrow v^i\perp v^j.$$
\end{enumerate}
\end{lemma}

\begin{proposition} \label{vectores en S}
	Let $\{v^j\}_{j=1}^{\infty}$, be the principal vectors defined in \eqref{vector ppal}. Then the following statements hold.
	\begin{enumerate}
		\item Given $w\in S$, with $\|w\|=1$. Then, for every $j$, 	$$w_j=\|P_Se_j\|\left\langle w,v^j\right\rangle $$ 
		and $|w_j|\leq v_j^j= |v_j^j|$.
		\item \label{item2propVecPpales}  $v^j_j=|w_j|$ if and only if $w=e^{i\arg(w_j)}v^j. $		
		\item \label{item3propVecPpales} In particular, $v^j_j=|v_j^k|$ if and only if $v^k=e^{i\arg(v^k_j)}v^j$. This is also equivalent to $|v_i^j|=|v_i^k|$ for every $i\in \N$.
		\item \label{item4propVecPpales} As a consequence, $\{v^j,v^k\}$ is linearly independent if and only if
		$$v^j_j\neq|v_j^k|\Leftrightarrow v^k_k\neq|v_k^j|$$
	\end{enumerate}
\end{proposition}

\begin{proof}
Let $w\in S$ with $\|w\|=1$, then there exists $v\in H$ such that $w=P_Sv$ and by Lemma \ref{propi ps y vectores ppales}
$$w_j=\left\langle w,e_j\right\rangle= \left\langle P_Sv,e_j \right\rangle=\left\langle w,P_Se_j\right\rangle=\|P_Se_j\|\left\langle w, v^j \right\rangle.$$
On the other hand, observe that $v^j_j=|w_j|$ yields to
$$v^j_j=\|P_Se_j\|\  |\left\langle w, v^j \right\rangle|,$$
or equivalently
$$\|w\|\|v\|=1=\left\langle v^j, v^j \right\rangle=|\left\langle w, v^j \right\rangle|.$$
Then, equality of Cauchy-Schwarz is attained if and only if $w$ and $v^j$ are multiples, that is $w=\lambda v^j$ with $|\lambda|=1$. 

Item \eqref{item3propVecPpales} can be proved replacing $w=v^k$ in item $\eqref{item2propVecPpales}$.
\end{proof}

The following result can be proved as a consequence of Proposition \ref{vectores en S}, using the same arguments that in Proposition 4.4 in \cite{kloboukvarela-mom-jnr}.

\begin{proposition}\label{extreme points in ms}
Let $S$ be a generic subspace of $H$. Then, $|v^j|^2=(|v^j_1|^2, |v^j_2|^2,\dots, |v^j_n|^2,\dots) $ 
is an extreme point in $m_S$. 
Moreover, if $|v^j|^2 $ is a convex combination of $|y|^2$ and $|z|^2$ with $y, z\in S$, then $y$ and $z$ must be multiples of $v^j$.
\end{proposition}
 
\subsection{Curves of extreme points in $m_S$}
\begin{definition}\label{curvas}
	Let $S$ be a generic subspace of $H$ and $v^j,v^k$ two linear independent principal standard vectors of $S$. We define the curve, $v^{j\to k}:[0,2\pi]\to S$
	\begin{equation}\label{eq curvas}
		\vjk=\cos(t)v^j+\sin(t)e^{i\arg(v_k^j)}\dfrac{(v^k-\left\langle v^k,v^j \right\rangle v^j )}{\|v^k-\left\langle v^k,v^j \right\rangle v^j\|},
	\end{equation}
\end{definition}

Next, we establish some properties of these curves in analogy with \cite{kloboukvarela-mom-jnr}. They can be proved using standard techniques.

\begin{proposition} \label{prop of curves}
	Let $S$ be a generic subspace of $H$ with $\{v^j\}_{j\in \N}$ the collection of principal unitary vectors related to the standard basis $E$ and $S$. The following properties hold:
	\begin{enumerate}
		\item The vectors $v^j$ and $\dfrac{(v^k-\left\langle v^k,v^j \right\rangle v^j )}{\|v^k-\left\langle v^k,v^j \right\rangle v^j\|}$ are unitary and orthogonal. Then $\|\vjk\|=1$ for every $t$, $v^{j\to k}(0)=v^j$ and
		$$\left\langle \vjk,e^{i\arg\left( v^j_k\right)}v^k \right\rangle\geq 0, \text{ for every } t\in \left[0,{\pi}/{2}\right]. $$
		\item By Lemma \ref{propi ps y vectores ppales} the $j$ and $k$ coordinates of $\vjk$ are
		\begin{equation}\label{coord vkj}
			v^{j\to k}_j(t)=\cos(t)v_j^j\ \text{ and }\ v^{j\to k}_k(t)=\cos(t)v^j_k +\sin(t)e^{i\arg\left( v^j_k\right)}\sqrt{(v_k^k)^2-|v_k^j|^2},
		\end{equation}
		respectively.
		\item The restriction $\vjk:\left[0,\frac{\pi}{2}\right]\to \text{Im } (v^{j\to k})$ is bijective. 
		\item If $\beta^{j\to k}(t)=\cos(t)\dfrac{e_j}{\|P_Se_j\|}+\sin(t)e^{i\arg(v_k^j)}\dfrac{\left( \dfrac{e_k}{\|P_Se_k\|}-\left\langle v^k,v^j \right\rangle \dfrac{e_j}{\|P_Se_j\|}\right) }{\|v^k-\left\langle v^k,v^j \right\rangle v^j\|}$,
		then 
		$$\vjk=P_S(\beta^{j\to k}(t)).$$
		\item If $e^{j\to k}(t)=\dfrac{\beta^{j\to k}(t)}{\left\| \beta^{j\to k}(t)\right\| }$, then
		$$\left\langle \vjk,e^{j\to k}(t)\right\rangle=\max_{s\in S,\|s\|=1}|\left\langle s,e^{j\to k}(t)\right\rangle|=\left\| P_S(e^{j\to k}(t)) \right\|$$
		and $\vjk=\dfrac{e^{j\to k}(t)}{\left\| e^{j\to k}(t)\right\| }$.
		\item If $w\in S$ with $\|w\|=1$,
		$$\left| \left\langle \vjk,e^{j\to k}(t)\right\rangle\right|= \left\langle \vjk,e^{j\to k}(t)\right\rangle\geq |\left\langle w,e^{j\to k}(t)\right\rangle|,$$
		for all $t\in\left[0,\frac{\pi}{2}\right] $. Moreover, 
		\begin{equation}
			\left\langle \vjk,e^{j\to k}(t)\right\rangle=| \left\langle w,e^{j\to k}(t)\right\rangle| \Leftrightarrow w=e^{i\arg\left(  \left\langle w,e^{j\to k}(t)\right\rangle\right) }\vjk.
		\end{equation}
		\item In particular, 
		$$\left\langle \vjk,e^{j\to k}(t)\right\rangle=| \left\langle v^{j\to k}(t_0),e^{j\to k}(t)\right\rangle|,\text{ for } t_0\in \left[0,{\pi}/{2}\right]$$
		$$
		\Leftrightarrow |\left\langle \vjk,e^{j\to k}(t)\right\rangle|=| \left\langle v^{j\to k}(t_0),e^{j\to k}(u)\right\rangle|,\ \forall u\in \left[0,{\pi}/{2}\right].
		$$
		\item As a consequence, the set $\{\vjk,v^{j\to k}(s)\}$ is linearly independent if and only if 	 
		$$\left\langle \vjk,e^{j\to k}(t)\right\rangle  \neq | \left\langle v^{j\to k}(s),e^{j\to k}(t)\right\rangle|$$
	\end{enumerate}
\end{proposition}

\begin{theorem} 
	If $\vjk$ is the curve defined in \eqref{eq curvas}, with $t\in \left[0,\frac{\pi}{2}\right]$, and $x\in S$ with $\|x\|=1$. Then, there exists a unique $t_x\in \left[0,\frac{\pi}{2}\right]$ such that
	\begin{equation}\label{tx}
		|x_j|=|v_j^{j\to k}(t_x)|\ \text{ and }\ |x_k|\leq |v_k^{j\to k}(t_x)|.
	\end{equation} 
	Moreover, if 
	\begin{equation}\label{wjk}
		w^{jk}=e^{i\arg(v_k^j)}\dfrac{(v^k-\left\langle v^k,v^j \right\rangle v^j )}{\|v^k-\left\langle v^k,v^j \right\rangle v^j\|}
	\end{equation}
	and $x=av^j+b w^{jk}+c y$ with $y\in S$ and $y$ is orthogonal to $v^j$ and $w^{jk}$,  then $t_x=\arccos(|a|)$.
\end{theorem}

\begin{proof}
	The proof is analogous to the finite dimensional case presented in \cite[Theorem 5.5]{kloboukvarela-mom-jnr}. 
\end{proof}

\begin{theorem}\label{teo curva entre extremales gral}
	Let $S\subset H$ be a generic subspace, $\{v^j,v^k\}$ two linearly independent principal standard vectors, $m_S$ the moment of $S$ as in Proposition \ref{prop: equivalencias de momento}, and $\gamma_{j,k}:\left[0,\frac{\pi}{2}\right]\to m_S$, the curve defined by
	\begin{equation}\label{curvaextremal}
		\gamma_{j,k}(t)=\left|\vjk \right|^2=\left( |v_1^{j\to k}(t)|^2,  |v_2^{j\to k}(t)|^2,...\right) 
	\end{equation}
	with $\vjk$ as in \eqref{eq curvas}. Then, 
	\begin{enumerate}
		\item $\left( |v_j^{j\to k}(t)|,  |v_k^{j\to k}(t)|\right) $ is part of an ellipse in $\R^2$ centered at the origin.
		\item If $v^j$ and $v^k$ are orthogonal, then $\left( |v_j^{j\to k}(t)|^2,  |v_k^{j\to k}(t)|^2\right) $ parametrizes a segment that is in the boundary of the projection of $m_S$ onto the plane spanned by $e_j$ and $e_k$.
		\item \label{item3 curvas extremales} If $v^j$ and $v^k$ are not orthogonal, then  $\left( |v_j^{j\to k}(t)|^2,  |v_k^{j\to k}(t)|^2\right)$ is an extreme point in the set $\{(x_j,x_k):\ x\in m_S\}$ and $\gamma_{j,k}(t)$ is an extremal point of $m_S$ for every $t\in \left[0,\frac{\pi}{2}\right]$.
	\end{enumerate}
	
\end{theorem}

\begin{proof}
	Using the coordinates of $\vjk$ given in \eqref{coord vkj}, it is evident that the pair
	$\left( |v_j^{j\to k}(t)|,  |v_k^{j\to k}(t)|\right)$ is part of an ellipse centered at $(0,0)$ for $t\in \left[0,\frac{\pi}{2}\right]$.
	
	On the other hand, observe that 
	\begin{equation}\label{curva extremal en r2}
		\begin{split}
			\left( |v_j^{j\to k}(t)|^2,  |v_k^{j\to k}(t)|^2\right)= & \cos^2(t)\left((v_j^j)^2,|v_k^j|^2 \right)+\sin^2(t)\left(0,(v_k^k)^2-|v_k^j|^2 \right)\\
		&+2\sin(t)\cos(t)\left(0, |v_k^j|\sqrt{(v_k^k)^2-|v_k^j|^2} \right).
		\end{split}
	\end{equation}
	Note that $\left( |v_j^{j\to k}(0)|^2,  |v_k^{j\to k}(0)|^2\right)=\left((v_j^j)^2,|v_k^j|^2 \right)$ and  $\left( |v_j^{j\to k}\left(\frac{\pi}{2} \right) |^2,  |v_k^{j\to k}\left( \frac{\pi}{2}\right) |^2\right)=\left(0,(v_k^k)^2-|v_k^j|^2 \right)$. So, there are different cases of this curve to explore: 
	
	a) The last term in \eqref{curva extremal en r2} is $0$ only if $t\in \{0,\pi/2
	\}$, $v_k^j=0$ or $v_k^k=|v_k^j|$. This last condition cannot hold, since $\{v^j,v^k\}$ are linearly independent by hypothesis (see item \eqref{item2propVecPpales} in Proposition \ref{vectores en S}). Then, \eqref{curva extremal en r2} is a segment only when $v_k^j=0$. In this case, 
	$$0=v_k^j=\left\langle v^j,e^k\right\rangle= \|P_Se_k\|\left\langle v^j,v^k\right\rangle, $$
	that is $v^j\perp v^k$.
	
	b) Now,  $v^j$ and $v^k$ are not orthogonal if and only if $v^j_k\neq 0$. Then the curve given by \eqref{curva extremal en r2} can be viewed as the graph of a map $f:[0,(v_j^j)^2]\to (0,+\infty)$ that is concave. Hence, using \eqref{tx} it can be proved that $\left( |v_j^{j\to k}(t)|^2,  |v_k^{j\to k}(t)|^2\right)$ is an extreme point in the set $\{(x_j,x_k):\ x\in m_S\}\subset \R^2$. Using this last fact and following the same steps than in  \cite[Theorem 5.6]{kloboukvarela-mom-jnr}, it can be proved that $\gamma_{j,k}(t)$ is an extremal point of $m_S$ for every $0\leq t\leq \frac{\pi}{2}$. 
\end{proof}

\begin{remark} 
As seen in Remark \ref{rem aff hull mS} the affine hull of $m_S$ is finite dimensional if $\dim(S)<\infty$. Nevertheless, the extremal curves $\gamma_{j,k}$ mentioned in \eqref{item3 curvas extremales} of Theorem \ref{teo curva entre extremales gral} might still be different for infinite pairs $j, k\in \N$. The following results give a more precise idea of these situation.
\end{remark}
\begin{theorem}
		Let $S\subset H$ be a generic subspace, $\{v^j,v^k\}$ two linearly independent principal standard vectors with $v^j\not\perp v^k$,  and $\gamma_{j,k}:\left[0,\frac{\pi}{2}\right]\to m_S$ a curve defined as in \eqref{curvaextremal}. Then, if $\gamma_{m,n}$ is another curve of the form \eqref{curvaextremal} with $\{v^m,v^n\}$ linearly independent satisfying $\gamma_{j,k}(t_0)=\gamma_{m,n}(t_1)$, then 
		$$
		\text{
		either } \left( v^j=v^m \wedge v^k=v^n \right) \text{ or } \left( v^j=v^n \wedge v^k=v^m \right).
		$$
 \end{theorem}
 \begin{proof}
 	The proof follows applying similar techniques as the ones used in 
 	 \cite[Theorem 5.6]{kloboukvarela-mom-jnr} in order to prove that the points $\gamma_{j,k}(t_0)$ are extremal. More precisely, if we suppose that $\gamma_{j,k}(t_0)=\gamma_{m,n}(t_1)$ it can be proved that  $\gamma_{j,k}(t_0)=|v^m|^2=|v^n|^2$ holds. Then using \eqref{item3propVecPpales} of Proposition \ref{vectores en S} this contradicts the supposition that $v^m$ and $v^n$ are linearly independent.
 \end{proof}
\begin{corollary}
	Let $S\subset H$ be a generic subspace, $\{v^j,v^k\}$ and $\{v^m,v^n\}$ two pairs of linearly independent principal standard vectors with $v^j\not\perp v^k$. Then $\gamma_{j,k}$ and $\gamma_{m,n}$ do not intersect each other.
\end{corollary}

\section{The moment $m_S$ and the space of Hermitian trace zero $\dim S\times \dim S$ matrices}\label{secc mS y traza cero}

In this section we show that the subalgebra $\mathcal{B}_S=P_S B(H) P_S$ of $\mathcal{K}(H)$ is isometrically isomorphic with the space of $r\times r$ complex matrices. 

Let $\{s^j\}_{j=1}^r$ be an orthonormal basis of $S$. Consider the standard basis in $\R^r$ given by $R=\{(1,0,\dots, 0),(0,1,0,\dots,0),\dots,(0,\dots,0,1) \}$ and denote these vectors with $e_1, e_2,\dots, e_r$ as usual. Using this prefixed basis, we will denote by $e_i\otimes e_j$ for all $1\leq i,j\leq r$ the $r\times r$ rank one matrices defined by 
$$
e_i\otimes e_j=e_i\cdot \left(e_j\right)^t , 
$$
where $e_k$ denotes the $k^\text{th}$ element of $R$ and $(e_k)^t$ its transpose.

We define the following sets
\begin{equation}\label{def calMr y calVS}
	\begin{split}
	\mathcal{M}_r^{h,0}&=\{M\in M_r(\C):\ M=M^*, \tr(M)=0\}\\
	\mathcal{V}_S^{h,0}&=\{A\in \mathcal{K}(H)^h:\ P_SA=A, \tr(A)=0\}.	
	\end{split}
\end{equation}
When the context is clear we will just denote them with  $\mathcal{M}_r$ and $\mathcal{V}_S$.
It is evident that $\mathcal{M}_r$ and  $\mathcal{V}_S$ are real subspaces of $M_r(\C)$ and $\mathcal{K}(H)$, respectively. 
Observe that for $\DD_S$ as in \eqref{def DsubS} 
$$m_S-\frac 1 r\Diag(P_S)=\Diag\left( \mathcal{D}_S-\frac 1 rP_S\right)$$
is a subset of $\mathcal{V}_S$ since for every $Y\in \mathcal{D}_S$ holds that $Y-\frac 1 rP_S\in \mathcal{K}(H)^h$, $\tr(Y-\frac 1 rP_S)=0$ and
$$
\text{aff}(m_S)-\frac 1 r\Diag(P_S)\subseteq \Diag(\mathcal{V}_S).
$$
\begin{proposition}
	Let $S$ be a finite dimensional subspace of $H$, $\DD_S$ as in \eqref{def DsubS} and $\mathcal{V}_S$ as in \eqref{def calMr y calVS}. Then the following equality holds
	$$
	\text{aff}(\mathcal{D}_S)-\frac1r P_S=\mathcal{V}_S
	$$ 	
	and as a consequence 
	$\text{Diag}\left(\text{aff}(\mathcal{D}_S)-\frac1r P_S\right) =\text{aff}(m_S)-\frac 1 r\Diag(P_S)= \Diag(\mathcal{V}_S)$.
\end{proposition}

\begin{proof}
	Take first $X=\sum_{i=1}^k a_i Y_i\in \text{aff}(\mathcal{D}_S)$ with $a_i\in\R$, $Y_i\in\mathcal{D}_S$  for all $i=1,\dots k$ and $\sum_{i=1}^k a_i=1$. Then $\tr(X-\frac1r P_S)=1-1=0$, $X-\frac1r P_S$ is hermitian and $P_S(X-\frac1r P_S)=X-\frac1r P_S$ which proves that $\text{aff}(\mathcal{D}_S)-\frac1r P_S\subset\mathcal{V}_S$.
	\\
	To prove the other inclusion let $Z\in \mathcal{V}_S$. Then $\tr(Z)=0$, $Z^*=Z$, and consider $Z=\sum_{i=1}^r \lambda_i (v^i\otimes v^i)$ a spectral decomposition of $Z$ with $\sum_{i=1}^r\lambda_i=0$, $v^i\in S$, $\|v^i\|=1$ and $v^i\perp v^j$ for $i\neq j$. Then since $P_S=\sum_{i=1}^r v^i\otimes v^i$
	$$
	Z=\sum_{i=1}^r \lambda_i (v^i\otimes v^i)+\frac1r P_S-\frac1r P_S=
	\sum_{i=1}^r \left(\lambda_i+\frac1r\right) (v^i\otimes v^i)-\frac1r P_S.
	$$
	Observe that if $Y_i=v^i\otimes v^i$, for $i=1,\dots,r$, then $\tr(Y_i)=1$, $0\leq Y_i\in S$ and hence $Y_i=v^i\otimes v^i\in\mathcal{D}_S$. Moreover, $\sum_{i=1}^r (\lambda_i+\frac1r)=0+1=1$ and hence $Z\in \text{aff}(\mathcal{D}_S)-\frac1r P_S$.
	
	The  equality $\text{Diag}\left(\text{aff}(\mathcal{D}_S)-\frac1r P_S\right) = \Diag(\mathcal{V}_S)$ follows using the linearity of $\text{Diag}$ and the fact that $\text{Diag}(\mathcal{D}_S)=m_S$.
\end{proof}

Now define the following $r\times r$ hermitian matrices with zero trace of $\mathcal{M}$
$$
W^{j,j}=\frac{1}{\sqrt{1+1/j}} \left(\left(\sum_{l=1}^{j}\frac{1}{j} e_{l}\otimes e_l\right)-e_{(j+1)}\otimes e_{(j+1)}\right) , \text{ for } j=1,\dots, r-1,
$$
$$
W^{k,j}=\frac{1}{\sqrt{2}}\left( e_{k}\otimes e_j +e_{j}\otimes e_{k}\right), \text{ for } k,j=1,\dots, r \text{ and } k< j,
$$

\begin{equation}\label{def mat Wij}
	W^{k,j}=	\frac{i}{\sqrt{2}}\left( e_{k}\otimes e_j - e_{j}\otimes e_{k}\right), \text{ for } k,j=1,\dots, r \text{ and } j< k
\end{equation}

and the trace zero self-adjoint operators of $\mathcal{V}$ obtained using an orthonormal basis $\{s^l\}_{l=1}^r$ of $S$
$$
V^{j,j}=\frac{1}{\sqrt{1+1/j}} \left(\left(\sum_{l=1}^{j}\frac{1}{j} s^{l}\otimes s^l\right)-s^{(j+1)}\otimes s^{(j+1)}\right) , \text{ for } j=1,\dots, r-1,
$$

$$V^{k,j}= \frac{1}{\sqrt{2}}\left(s^k\otimes s^j+s^j\otimes s^k\right), \text{ for } k,j=1,\dots, r \text{ and } k< j$$

\begin{equation}\label{def operadores Vij}
	V^{k,j}=	\frac{i}{\sqrt{2}}\left(s^k\otimes s^j-s^j\otimes s^k\right), \text{ for } k,j=1,\dots, r \text{ and } j< k.
\end{equation}

Then, for the set $J=\lbrace (k,j): k=1,\dots,r \wedge j=1,\dots, r\rbrace\setminus \lbrace (r,r)\rbrace$, easy calculations show that 
$$\lbrace W^{k,j}\rbrace_{(k,j)\in J}\ \text{ and }\ \lbrace V^{k,j}\rbrace_{(k,j)\in J}$$
are real orthonormal basis for $\mathcal{M}_r$ and $\mathcal{V}_S$ respectively (taking the inner product given by the trace in both cases), and both subspaces have $\dim=r^2-1$. The set $\lbrace W^{k,j}\rbrace_{(k,j)\in J}$ without the normalization is known as the generalized Gell-Mann basis \cite{BK}.

\begin{remark}\label{rem bon Wij + Ir Vij + Ps}
	Let $S$ be a finite dimensional subspace of $H$ with a fixed orthonormal basis $\{s^j\}_{j=1}^r$.
	Observe that, with the notations presented in the previous discussion, the set $\lbrace W^{k,j}\rbrace_{(k,j)\in J}\cup \left\lbrace \frac{1}{\sqrt{r}} I\right\rbrace $ is a real orthonormal basis of $M_r^h(\C)$ and also a complex orthonormal basis of $M_r(\C)$, that is
	$$\text{span}\left\lbrace \frac{1}{\sqrt{r}}I_r\right\rbrace\oplus_{\R}\text{span} \lbrace W^{k,j}\rbrace_{(k,j)\in J}=M_r^h(\C),$$
	and 
	$$\text{span}\left\lbrace \frac{1}{\sqrt{r}}I_r\right\rbrace\oplus_{\C}\text{span} \lbrace W^{k,j}\rbrace_{(k,j)\in J}=M_r(\C).$$
	On the other hand, the subspace $\text{span}\left\lbrace P_S\right\rbrace \oplus_{\C}\mathcal{V}_S$ is a subalgebra of $\mathcal{K}(H)$, and it can be identified with  $\mathcal{B}_S=P_S B(H) P_S=\text{span}\left\lbrace P_S\right\rbrace \oplus_{\C}\mathcal{V}_S$. In this context $ \{\frac1{\sqrt{r}} P_S\} \cup \lbrace V^{k,j}\rbrace_{(k,j)\in J}$ is also an orthonormal basis (respect the trace inner product) of the real subspace $\mathcal{B}^h_S$ of its hermitian operators.
\end{remark}

\begin{proposition} \label{props unitario U}
	Using the previous notations we define the bijective linear operator $U:M_r(\C)\to   \mathcal{B}_S$
	on the orthonormal matrices defined \eqref{def mat Wij} and $\frac{I_r}r$ in the following way
	$$\left\lbrace \begin{array}{lllll}
		U(W^{k,j})&=&V^{k,j}&\text{for every}&(k,j)\in J\\
		U\left( 
		I_r\right) &=& 
		P_S,& &
	\end{array}\right. 
	$$
	where the operators $V^{k,j}$ are defined in \eqref{def operadores Vij}.
	
	Then, for every $A,B\in M_r(\C)$
	\begin{enumerate}
		\item \label{prop1deU} $\tr(U(A))=\tr(A)$.	
		\item  \label{prop2deU} $\left( U(A)\right)^*= U(A^*)$.
		\item $\left( U(A)\right)^*=U(A)$ if and only if $A=A^*$. 
		\item  \label{prop4deU}$U(AB)=U(A)U(B)$ and $U^{-1}(U(A) U(B))=A B$.
		\item If $A\in M_r(\C)$ is invertible, then $U(A)$ is invertible in the algebra $\mathcal{B}_S$  and  $U(A^{-1})U(A)=P_S$.
		\item \label{prop5deU} $A\geq 0$ if and only if $U(A)\geq 0$.
		\item \label{prop6deU}
		$\left\langle U(A),U(B) \right\rangle_{tr}=\tr\left( U(A)(U(B))^*\right) =\tr\left( AB^*\right)=\left\langle A,B \right\rangle_{M_r(\C)} $ ($U$ is unitary).
		\item \label{prop8deU} $P\in M_r(\C)$ is a projection if and only if $U(P)$ is a projection.
		\item \label{prop9deU} $U\left(\{R\in M_n^h(\C): R\geq 0 \wedge \tr(R)=1\}\right) = \DD_S$ (with $\DD_S$ as in \eqref{def DsubS}).
	\end{enumerate}
\end{proposition}

\begin{proof}	
	First observe that any $A\in M_r(\C)$ can be written in terms of the orthonormal basis defined in Remark \ref{rem bon Wij + Ir Vij + Ps}:
	$$
	A=a_{r}I_r+\sum_{(k,j)\in J} a_{kj}W^{k,j},\text{ with } a_{r},a_{kj}\in \C.$$
	Then, 
	$$U(A)=U\left( a_{r}I_r+\sum_{(k,j)\in J} a_{kj}W^{k,j}\right)= a_rP_S+\sum_{(k,j)\in J} a_{kj}V^{k,j} $$
	\begin{enumerate}
		\item $\tr(U(A))=\tr\left(a_rP_S+\sum_{(k,j)\in J} a_{kj}V^{k,j} \right)=a_r r+\sum_{(k,j)\in J} a_{kj}\tr(V^{k,j})=a_rr=\tr(A) $.
		\item The result is obvious since $\left( U(A)\right)^* =\overline{a}_rP_S+\sum_{(k,j)\in J} \overline{a}_{kj}V^{k,j}$ and $A^*=\overline{a}_rI_r+\sum_{(k,j)\in J} \overline{a}_{kj}W^{k,j}$.
		\item If $A=A^*$, it is a direct consequence from item \eqref{prop2deU}  that $\left( U(A)\right)^*=U(A)$. On the other hand, if $\left( U(A)\right)^*=U(A)$, then $U(A)=U(A^*)$ and
		$$a_rP_S+\sum_{(k,j)\in J} a_{kj}V^{k,j}=\overline{a}_rP_S+\sum_{(k,j)\in J} \overline{a}_{kj}V^{k,j}$$	
		which means that $a_rr,a_{kj}\in \R$. Therefore, $A=A^*$. 
		\item According to (4) in \cite{BYK} there exist complex coefficients $\alpha_r$ and $\alpha_{ll'}$ such that every product of elements of $\{W^{k,j} \}_{(k,j)\in J}$ can be written as
		$$W^{k,j}W^{k',j'}=\alpha_{r}I_r+\sum_{(l,l')\in J} \alpha_{k,j,k',j',l,l'}\ W^{l,l'},$$
		\\
		and similarly
		\\
		$$V^{k,j}V^{k',j'}=\alpha_{r}P_S+\sum_{(l,l')\in J} \alpha_{k,j,k',j',l,l'}\ V^{l,l'}
		$$
		for $(l,l')\in J$,	with the same coefficients $\alpha_{k,j,k',j',l,l'}\in\C$. This follows considering the definitions \eqref{def mat Wij} and \eqref{def operadores Vij} and the orthonormality of the basis $\{e_l\}_{l=1}^r$ and $\{s\}_{l=1}^r$.
		Then, 
		\begin{equation*}
			\begin{split}
				U\left(W^{k,j}W^{k',j'}\right)&=
				\alpha_{r}U\left(I_r\right)+\sum_{(l,l')\in J} \alpha_{k,j,k',j',l,l'}\ U\left(W^{l,l'}\right)=\alpha_{r}P_S+\sum_{(l,l')\in J} \alpha_{k,j,k',j',l,l'}\ V^{l,l'}\\
				&=V^{k,j}V^{k',j'}=U\left(W^{k,j}\right)\, U\left(W^{k',j'}\right).		
			\end{split}
		\end{equation*}
		Then, applying this property, the fact that $\{W^{k,j}\}_{(k,j)\in J} \cup \{\frac{I_r}{r}\}$ is an orthonormal basis of $M_r(\C)$ and the linearity of $U$ imply that $U(A B)=U(A)U(B)$ for all $A, B\in M_r(\C)$.
		
		The equality $A B=U^{-1}\left(U(A)U(B)\right)$ follows similarly.
		
		\item Follows directly from item \eqref{prop4deU}, since $U(A^{-1})U(A)=U(A^{-1}A)=U(I_r)=P_S$.
		
		\item If $A\geq 0$ there exists $T\in M_r(\C)$ such that $A=T^*T$. Then, using items 3 and 4
		$$U(A)=U(T^*T)=\left( U(T)\right)^*U(T)\geq 0.$$
		On the other hand, if $U(A)\geq 0$, then there exists $K\in \mathcal{K}(H)$ such that $U(A)=K^*K$. Moreover, since $\mathcal{B}_S$ is a subalgebra $K\in \mathcal{B}_S$. Then, $K=U(B)$ with $B\in M_r(\C)$,
		$$U(A)=U(B)^*U(B)=U(B^*B),$$
		and $A=B^*B\geq 0$.
		\item Using that $\{W^{k,j}\}_{(k,j)\in J}$ and $\{V^{k,j}\}_{(k,j)\in J}$ are orthonormal sets of zero trace, then
		\begin{eqnarray*}
			\tr\left( U(A)U(B)^*\right)&=&\tr\left(U\left( a_rI_r+\sum_{(k,j)\in J} a_{kj}W^{k,j}\right)U\left( \bar b_rI_r+\sum_{(k,j)\in J} \bar b_{kj}W^{k,j}\right) \right)\\
			&=& \tr\left(\left(  a_rP_S+\sum_{(k,j)\in J} a_{kj}V^{k,j}\right) \left(\bar b_rP_S+\sum_{(k,j)\in J} \bar b_{kj}V^{k,j} \right) \right)\\
			&=& \tr\left(a_r\bar b_rP_S+\sum_{(k,j)\in J} a_{kj}\bar b_{kj}V^{k,j}\right)\\
			&=& a_r\bar b_r =\tr(AB^*).
		\end{eqnarray*}
		The items \eqref{prop8deU} and \eqref{prop9deU} can be proved easily using the previous items \eqref{prop1deU}, \eqref{prop4deU}, \eqref{prop5deU} and \eqref{prop6deU}.

	\end{enumerate}
	
\end{proof}

\begin{remark}
	The restriction $\left. U\right|_{M_r^h(\C)}$ is a (real) isometric isomorphism between $M_r^h(\C)$ and $\mathcal{B}_S^h=P_S B(H)^hP_S=\text{span}\left\lbrace P_S\right\rbrace \oplus_{\R}\mathcal{V}_S$. Additionally, 
	$\left. U\right|_{\mathcal{M}_r}$ is an isometry between $\mathcal{M}_r$ and $\mathcal{V}_S$.
	
\end{remark}

\begin{corollary}\label{prop JNR escrito con matrices finitas} 
	With the same notations of the previous paragraphs, the following two joint numerical ranges coincide
	$$
	W(P_S E_1 P_S,\dots,P_S E_n P_S)=W\left( 
	{U^{-1}(P_S E_1 P_S)}
	,\dots,U^{-1}(P_S E_n P_S)\right), \ \forall n\in \N.
	$$
\end{corollary}
\begin{proof}
	The proof follows directly from properties \eqref{prop1deU}, \eqref{prop4deU}, \eqref{prop5deU} and \eqref{prop9deU} of Proposition \ref{props unitario U}.
\end{proof}

\begin{remark}
	In the finite dimensional case, a similar result as the one in Corollary \ref{prop JNR escrito con matrices finitas} can be obtained as mentioned in Remark 6.3 (3) of \cite{kloboukvarela-mom-jnr}. 
    In that description the joint numerical ranges of a subspace $S\subset \C^n$ are related with joint numerical ranges of $\dim(S)\times \dim(S)$ matrices.
\end{remark}
\section{Condition of minimality using finite $n\times n$ matrices}\label{secc minimalidad con finitas mat}
Let $S$, $V$ be orthogonal subspaces of $H$ with dim$(S)=r$ and dim$(V)=t$.
In this section we will use the operators $U_S:M_r(\C)\to \mathcal{B}_S$ and $U_V:M_t(\C)\to \mathcal{B}_V$ defined in Proposition \ref{props unitario U} to relate some properties of $S$ and $V$ with the more manageable case of $r\times r$ and $t\times t$ hermitian matrices.

For every $q\in \N$, we define the real functionals $\varphi_q: {M}^h_r(\C)\to \R$ by
$$
\varphi_q(M)=\left\langle U(M)e_q,e_q\right\rangle=\left(U(M)_{E,E}\right)_{q,q}
$$
(the $q,q$ diagonal entry of $U(M)$ considering the standard basis $E$).

By the Dimension Theorem, $\dim_{\R}(\ker(\varphi_q))= r^2-1 $ and hence $\dim(\ker(\varphi_q)^{\perp})=1$. Therefore, $\varphi_q$  can be written as
$$\varphi_q(M)=\left\langle M,Q_q\right\rangle_{tr}=\tr(Q_q M),$$
with some $Q_q\in \ker(\varphi_q)^{\perp}\subset M^h_r(\C)$ and $\|Q_q\|_2=\sqrt{\tr \left( (Q_q)^2 \right)}=1$.

Now suppose $M\in {M}_r^h(\C)$ is written as $M=a_{r,r} \frac{I_r}{\sqrt{r}}+\sum_{(k,j)\in J} a_{k,j} W^{k,j}$, where $a_{k,j}\in\R$ are its coordinates in the orthonormal basis of the real space ${M}_r^h(\C)$ (see Remark  \ref{rem bon Wij + Ir Vij + Ps}).
Then, for $q\in\N$,
\begin{equation}
	\begin{split}
		\varphi_q(M)&=\left\langle U(M)e_h,e_h\right\rangle_H =(U(M)_{E,E})_{h,h}
		=a_{r,r}\left(U\left(\frac{I_r}{\sqrt{r}}\right)\right)_{h,h}+ \sum_{(k,j)\in J} a_{k,j} \left(U(W^{k,j})_{E,E}\right)_{h,h}
		\\
		&=a_{r,r}\left(\frac{P_S}{\sqrt{r}}\right)_{h,h}+\sum_{(k,j)\in J} a_{k,j} \left(V^{k,j}_{E,E}\right)_{h,h}=\langle M, Q_q\rangle
	\end{split}
\end{equation}
for $Q_q=\left(\frac{P_S}{\sqrt{r}}\right)_{q,q} \frac{I_r}{\sqrt{r}} +
\sum_{(k,j)\in J} \left(V^{k,j}_{E,E}\right)_{q,q} W^{k,j}
\in  {M}_r^h(\C)$.
\\
Note that the vector $Q_q$  cannot be null since we are supposing that the subspace $S$ is generic (otherwise the $h,h$ coordinate in the $E$ basis would be $0$ for every operator in $S$).
Therefore, for $e_q\in E$ (standard basis in $K(H)$) 
$$
(U(M)_{E,E})_{q,q}=\left\langle U(M)e_q,e_q\right\rangle=\tr(Q_q M).
$$
Then, we can define $\varphi: {M}^h_r(\C)\to \Diag(B^h(S))\subset  \ell^1(\R)$ as $\varphi(M)=\Diag( U(M))$ and calculate it using
$$
\varphi(M)=(\varphi_1(M), \varphi_2(M),\dots,\varphi_q(M),\dots)=(\tr(Q_1 M),\tr(Q_2 M),\dots,\tr(Q_q M),\dots).
$$	

\subsection{Intersection of joint numerical ranges in terms of families with a finite number of operators}

Let $S$ an $r$-dimensional subspace of $H$ as before, and consider $\mathcal{B}_S^h$, with $\dim_\R (\mathcal{B}_S^h)=r^2$. Then define 
$\phi:\mathcal{B}_S^h\to \Diag(\mathcal{B}_S^h)\subset K^h(H)$ as
$\phi(A)=\Diag(A)$, where $\Diag$ is the diagonal in the standard $E$ basis of $H$. Note that since $S\subset H$ is finite dimensional then we can consider $\Diag(\mathcal{B}_S^h)\subset  \ell^1(\R)$. 

In this context, since $\phi_n(A)=A_{n,n}$ (the $n,n$ entry of $\Diag(A)$) is a functional of the space $\mathcal{B}_S^h$, there exist operators $T_n\in \mathcal{B}_S^h$, with $\|T_n\|_2=1$, such that
\begin{equation}
	\label{def Tn}
	\phi(A)=\Diag\left(\{\tr(A T_n)\}_{n\in \N}\right)=\Diag(A).
\end{equation}
Similarly, for another subspace $V$ of $H$ that is orthogonal to $S$, with dim$(V)=t$ we can define
$\psi:\mathcal{B}_V^h\to \Diag(\mathcal{B}_V^h)\subset K^h(H)$ as
$\psi(C)=\Diag(C)$. And also in this case there exist operators $L_n\in \mathcal{B}_V^h$, with $\|L_n\|_2=1$, such that
\begin{equation}
	\label{def Ln}
	\psi(C)=\Diag\left(\{\tr(C L_n)\}_{n\in \N}\right)=\Diag(C).
\end{equation}

\begin{proposition}\label{prop Delta con finitos}
	Let $\phi$ as in \eqref{def Tn}, $\psi$ in \eqref{def Ln} and define $\Delta: \mathcal{B}_S^h\oplus \mathcal{B}_V^h\to \Diag(K^h(H))$ as
	\begin{equation}\label{def de Delta}
		\Delta(A,C)=\phi(A)-\psi(C),\ \text{ for } A\in  \mathcal{B}_S^h \text{ and } C\in \mathcal{B}_V^h.
	\end{equation}
	Then there exists (after a suitable reordering of the basis $E$) a finite subset of $\{(T_n,L_n)\}_{n\in\N}$ that we will denote with $\{(T_n,L_n)\}_{i=1}^m$ such that
	\begin{equation}\label{def (Tn,Ln) n de 1 a m}
		\begin{split}
			(A,C)\in \ker (\Delta) &\Leftrightarrow \Diag(A)=\Diag(C)\\
			& \Leftrightarrow (A,-C)\perp (T_n,L_n) , \forall n=1,\dots,m 		
		\end{split}
	\end{equation}
\end{proposition}
\begin{proof}
	The first equivalence follows directly from the definition of $\Delta$.
	
	On the other hand we have that $(A,C)\in \ker (\Delta) \Leftrightarrow \Diag(A)=\Diag(C) \Leftrightarrow (A,-C)\perp (T_n,L_n) , \forall n\in\N$. Therefore we only need to prove that (after reordering the basis $E$) there exist $\{(T_n,L_n)\}_{i=1}^m$ such that if $ (A,-C)\perp (T_n,L_n) , \forall n=1,\dots,m 
	$, then $(A,C)\in \ker (\Delta)$.				
	For this purpose, recall that since $S$ and $V$ are finite dimensional subspaces of $H$, then also $\mathcal{B}_S^h$ and $\mathcal{B}_V^h$ are finite dimensional $\R$-subspaces of $B^h(H)$. Hence $\dim\left(\mathcal{B}_S^h\oplus \mathcal{B}_V^h\right)= r^2+t^2<\infty$, and then $\dim\left(\text{span}\left(\{(T_n,L_n)\}_{n\in\N}\right)\right)\leq r^2+t^2$.
	To alleviate the notation, we can reorder the diagonal entries by conjugation of unitary operators obtained after permutation of the corresponding rows and columns of the identity matrix in the $E$ basis. After this we can suppose that $\{(T_n,L_n)\}_{n=1}^m$ is a finite basis of span$\left(\{(T_n,L_n)\}_{n\in\N}\right)$.
	Then, it is apparent that for $(A,C)\in \mathcal{B}_S^h\oplus \mathcal{B}_V^h$,  $(A,C)\perp \text{span}\left(\{(T_n,L_n)\}_{n\in\N}\right)$ if and only if $(X,Y)\perp \text{span}\left(\{(T_n,L_n)\}_{n=1}^n\right)$.
\end{proof}

\begin{remark}
	\label{rem Delta en las bases VsubS y VsubV}
	Observe that we can also describe $\Delta$ in terms of multiplication of matrices using the orthogonal basis $\mathcal{V}_S$ and $\mathcal{V}_V$ 
	$$
	\Delta(A,C)=\begin{pmatrix}
		[T_1]_{\mathcal{V}_S}& 		[L_1]_{\mathcal{V}_V}\\
		[T_2]_{\mathcal{V}_S}& 		[L_2]_{\mathcal{V}_V}\\
		\dots&\dots\\
		\vdots&\vdots
	\end{pmatrix}_{\infty\times (r^2+t^2)}
	\cdot
	\begin{pmatrix}
		[A]_{\mathcal{V}_S}\\
		- [C]_{\mathcal{V}_V}
	\end{pmatrix}_{ (r^2+t^2)\times 1},
	$$
	where we denoted with $[\ ]_{\mathcal{V}_S}$ and $[\ ]_{\mathcal{V}_V}$ the  coordinates of the corresponding hermitian operators in the basis $\mathcal{V}_S$ and $\mathcal{V}_V$ respectively (see Remark \ref{rem bon Wij + Ir Vij + Ps}).
\end{remark}
\begin{corollary}\label{coro Diags iguales sii perp finitos}
	Let $\{(T_n,L_n)\}_{n=1}^m$ be as in Proposition \ref{prop Delta con finitos} (see \eqref{def de Delta} and \eqref{def (Tn,Ln) n de 1 a m}). Then, for $A\in  B^h(S)$, $C\in B^h(V)$ 
	\begin{equation}
		\begin{split}
			\label{eq diags = con A,-C perp Tn,Ln finitos y diags finitas}
			\Diag(A)=\Diag(C)  &\Leftrightarrow (A,-C)\perp (T_n,L_n), \forall n=1,\dots,m\\
			&\Leftrightarrow A_{n,n}=C_{n,n}, \forall n=1,\dots,m .
		\end{split}
	\end{equation}
	
\end{corollary}
\begin{proof}
	This follows after observing that if $(A,-C)\perp (T_n,L_n)$ then $0=\tr(A T_n)+\tr(-C L_n)=A_{n,n}-C_{n,n}$ (see \eqref{def Tn}, \eqref{def Ln}, \eqref{def de Delta}). Hence $\Diag(A)=\Diag(C)$ if and only if $(A,-C)\perp (T_n,L_n)$ for all $n\in\N$ which in term is equivalent to $(A,-C)\perp (T_n,L_n)$ for $n=1,\dots, m$ after using Proposition \ref{prop Delta con finitos}.
\end{proof}

\begin{corollary}
	Let $\{(T_n,L_n)\}_{n=1}^m$ be as in Proposition \ref{prop Delta con finitos} (see \eqref{def de Delta} and \eqref{def (Tn,Ln) n de 1 a m}). The following statements are equivalent
	\begin{enumerate}
		\item [a) ] $\dim \left(\text{span}\left(\{(T_n,L_n)\}_{n=1}^m\right)\right)<\dim\left( \mathcal{B}_S^h\oplus \mathcal{B}_V^h\right)$
		
		\item[b) ] $\exists$ a not null pair $(A,C)\in  \mathcal{B}_S^h\oplus \mathcal{B}_V^h$ such that $\Diag(A)=\Diag(C) $.
	\end{enumerate}
\end{corollary}
\begin{proof} Recall that $\{(T_n,L_n)\}_{n=1}^m$ is a basis of $\ker(\Delta)^\perp=\{(A,C)\in \mathcal{B}_S^h\oplus \mathcal{B}_V^h: \Diag(A)=\Diag(C)\}$ (see \eqref{def de Delta}). Then note that the condition $m=\dim \left(\text{span}\left(\{(T_n,L_n)\}_{n=1}^m\right)\right)<\dim\left( \mathcal{B}_S^h\oplus \mathcal{B}_V^h\right)=r^2+t^2$ is equivalent to the existence of a not null hermitian $(A,C)\in \mathcal{B}_S^h\oplus \mathcal{B}_V^h$ where $A$ and $C$ share the same diagonal. The implication b) $\Rightarrow$ a) follows similarly. 
\end{proof}


Now we can state the following result. 
\begin{proposition}\label{prop varias diags finitas} With the notations of the previous paragraphs of this section the following statements are equivalent
	\begin{enumerate}
		\item \label{prop1finitosOp} $\exists$ a not null $ (X,Y)\in \mathcal{B}_S^+\oplus \mathcal{B}_V^+$ such that $\tr(X)=\tr(Y)=1$ and $(X,Y)\in \ker(\Delta)$ (for $\Delta$ as in \eqref{def de Delta}).
		\item \label{prop2finitosOp}  $\exists$ a not null $ (X,Y)\in \mathcal{B}_S^+\oplus \mathcal{B}_V^+$ such that $\tr(X)=\tr(Y)=1$ and $(X,Y)\perp \{(T_n,L_n)\}_{n=1}^m$, where span$\{(T_n,L_n)\}_{n=1}^m=\ker(\Delta)^\perp$ (with $(T_n,L_n)$ as in Proposition \ref{prop Delta con finitos}).
		\item  \label{prop3finitosOp} $\exists$ a not null $ (X,Y)\in \mathcal{B}_S^+\oplus \mathcal{B}_V^+$  such that $\tr(X)=\tr(Y)=1$ and $(X,Y)\perp \{(T_n,L_n)\}_{n\in \N}$ (see \eqref{def Tn}, \eqref{def Ln}).
		\item  \label{prop4finitosOp} $\exists$ a not null $ (X,Y)\in \mathcal{B}_S^+\oplus \mathcal{B}_V^+$ such that $\tr(X)=\tr(Y)=1$ and $ X_{n,n}=Y_{n,n}$ ($n,n$ diagonal entries in the basis $E$), for $n=1,\dots ,m=\dim\left(\ker(\Delta)^\perp\right)$.
		\item  \label{prop5finitosOp} $\exists$ a not null $ (X,Y)\in \mathcal{B}_S^+\oplus \mathcal{B}_V^+$ such that $\tr(X)=\tr(Y)=1$ and $\Diag(X)=\Diag(Y)$
		\item  \label{prop6finitosOp} $m_S\cap m_V\neq \emptyset$
		\item \label{prop7finitosOp}  $W(P_SE_1P_S,\dots, P_SE_iP_S,\dots)\cap W(P_V E_1 P_V,\dots,P_V E_j P_V,\dots)\neq \{0\}$
		\item  \label{equiv jnr en matrices} $W(P_SE_1P_S,\dots, P_SE_mP_S)\cap W(P_V E_1 P_V,\dots,P_V E_m P_V)\neq \{0\}$
	\end{enumerate}
\end{proposition}
\begin{proof}
	The equivalences of the first five items follow directly from the previous results Proposition \ref{prop Delta con finitos} and Corollary \ref{coro Diags iguales sii perp finitos}. The equivalences involving \eqref{prop6finitosOp} and \eqref{prop7finitosOp}  with the first four statements can be proved using Proposition \ref{prop: equivalencias de momento}. To prove that statement \eqref{prop4finitosOp} is equivalent to \eqref{equiv jnr en matrices}, use that \eqref{prop4finitosOp} implies \eqref{prop7finitosOp} and that \eqref{prop7finitosOp} apparently implies \eqref{equiv jnr en matrices}. The other implication can be obtained  observing that if \eqref{equiv jnr en matrices} holds then there exists $X\in B^+(S)$, $Y\in B^(V)$ with $\tr(X)=\tr(Y)=1$ such that $\tr(X P_S E_n P_S)=\tr(Y P_V E_n P_V)$, for $n=1,\dots, m$, which in turn implies that $\tr(XE_n)=\tr(YE_n)$ and hence $X_{n,n}=Y_{n,n}$ for $n=1,\dots,m$ (which is \eqref{prop4finitosOp}.).
\end{proof}

\subsection{Minimal matrices, moment of subspaces and joint numerical ranges in terms of finite matrices}

As before, we will consider two orthogonal finite dimensional subspaces $S$ with dim$(S)=r$ and $V$ with dim$(V)=t$ of $H$. We want to study relations between their moment sets and joint numerical ranges to similar sets but on the ambient of $M_r(\C)$ and $M_t(\C)$.
For that purpose consider the map
\begin{equation}
	\label{def map Z}
	Z:M_r(\C)\times M_t(\C)\to \mathcal{B}_S\oplus \mathcal{B}_V, \text{ such that } 
	Z(M,N)=U_S(M)+ U_V(N)
\end{equation}
where $U_S$ and $U_V$ are the applications defined in Proposition \ref{props unitario U} for the respective subspaces $S$ and $V$. Here we are considering on $M_r(\C)\times M_t(\C)$ the usual scalar product $\langle(M,N),(X,Y)\rangle=\tr(M X^*)+\tr(NY^*)$. Observe that $Z$ is invertible with $Z^{-1}(C,D)=(U_S^{-1}(C),U_V^{-1}(D))$. Also note that using the properties of $U_S$ and $U_V$ (see Proposition \ref{props unitario U}) the map $Z$ is an isometric isomorphism that preserves trace, inner products and positive definiteness in each entry (among many other properties).

Suppose that there exists $(M,N)\in M_r^+(\C)\times M_t^+(\C)$ 
such that 
$$
(M,N)\perp \{ (U_S^{-1}(T_n), U_V^{-1}(L_n)\}_{n=1}^m
$$
for $(T_n,L_n)$ as defined in \eqref{def (Tn,Ln) n de 1 a m} of Proposition \ref{prop Delta con finitos}. This holds if and only if $U_S(M)\in \mathcal{B}_S^+$ and $U_V(N)\in \mathcal{B}_V^+$ satisfy $(U_S(M),U_V(N))\perp (T_n,L_n)$ for $n=1,\dots, m$, which is equivalent to $\Diag(U_S(M))=\Diag(U_V(N))$ and to the fact that $m_S\cap m_V\neq \emptyset$ (see Proposition \ref{prop varias diags finitas}).
\begin{proposition}\label{prop union t psubmalfa igual a W y W}	Let $S$ be a subspace of $H$, $U_S$ defined as in Proposition \ref{props unitario U}, $m_S$ as in \eqref{def momento}, and $p_m:\ell^1(\R)\to\R^m$ the projection defined by $p_m\left(x_1,\dots,x_n,\dots\right)=(x_1,\dots,x_m)$. Then
	$$
	\bigcup_{\alpha \in [0,1]} \alpha \ p_m(m_S)= W\left(\{P_S E_j P_S\}_{j=1}^m\right)=W\left(\left\{U_S^{-1}(P_S E_j P_S)\right\}_{j=1}^m
	\right). 
	$$
	%
	%
\end{proposition}
\begin{proof}
	The equality between the joint numerical range of operators $W\left(\{P_S E_j P_S\}_{j=1}^m\right)$ and the other $W\left(\left\{U_S^{-1}(P_S E_j P_S)\right\}_{j=1}^m
	\right)$ of matrices holds because $U_S^{-1}$ preserves joint numerical ranges (see Corollary \ref{prop JNR escrito con matrices finitas}).
	
	Now let $x \in  \cup_{\alpha \in [0,1]} \alpha \, p_m(m_S)$. Then $x=\alpha  (\tr(\mu E_1),\dots,\tr(\mu E_m))$, with $\alpha \in[0,1]$ and $\mu\in\DD_S$ (see \eqref{def DsubS} and \eqref{def momento}).
	Now consider $\rho=\alpha \, \mu+(1-\alpha ) \frac{P_V}{\dim V}$, for $V\subset S^\perp$ and $0<\dim(V)<+\infty$. Then it is apparent that $\tau(\rho)=1$, $\rho\geq 0$ and $\tr(P_S \rho P_S E_i)=\tr(P_S \alpha  \mu P_S E_i)=\alpha \, \tr(\mu E_i)$, for $i=1,\dots,m$. Hence $x=\alpha  (\tr(\mu E_1),\dots,\tr(\mu E_m))=(\tr(P_S \rho P_S E_1),\dots, \tr(P_S \rho P_S E_m)) \in  W\left(\{P_S E_j P_S\}_{j=1}^m\right)$. 
	
	To prove the other inclusion observe that the case when $x= (0,\dots,0)$ can be obtained with $\alpha =0$. So let us suppose $x\in W\left(\{P_S E_j P_S\}_{j=1}^m\right)$ and $x$ is not null. Then $x =(\tr(P_S \rho P_S E_1),\dots, \tr(P_S \rho P_S E_m))\in  W\left(\{P_S E_j P_S\}_{j=1}^m\right)$ with $\rho\in \mathcal{B}_1(H), \tr(\rho)=1, \rho\geq 0$. Since $P_S \rho P_S\geq 0$ and $x$ is not null, then $0<\tr({P_S \rho P_S})\leq 1$ in this case. We can define $\mu =\frac{P_S \rho P_S}{\tr({P_S \rho P_S})}\in\DD_S$ and then 
	$$
	x=\tr({P_S \rho P_S})\left(\frac{P_S \rho P_S}{\tr({P_S \rho P_S})} E_1,\dots,\frac{P_S \rho P_S}{\tr({P_S \rho P_S})} E_m\right)=\alpha  (\mu E_1,\dots,\mu E_m),
	$$
	for $\alpha =\tr({P_S \rho P_S})\in(0,1]$ and $\mu\in\DD_S$. This concludes the proof.
\end{proof}

\begin{theorem}
	\label{prop equivs mS cap mV no vacio}
	Let $S$ and $V$ be orthogonal subspaces of $H$, with dim$(S)=r$, dim$(V)=t$,  $\{(T_n,L_n)\}_{n=1}^m$ a basis of $\ker(\Delta)$ (see \eqref{def de Delta} and \eqref{def (Tn,Ln) n de 1 a m}), $U_S$, $U_V$ defined in \eqref{def map Z} and in Proposition \ref{props unitario U}, and the projection $p_m:\ell^1(\R)\to\R^m$ defined by $p_m\left(x_1,\dots,x_n,\dots\right)=(x_1,\dots,x_m)$.
	
	 Then the following statements are equivalent
	\begin{enumerate}
		\item[(1) ] $m_S\cap m_V\neq \emptyset$.
		\item[(2) ] $p_m(m_S)\cap p_m(m_V)\neq \emptyset$.
		\item[(3) ] $\exists (M,N)=(U_S^{-1}(X),U_V^{-1}(Y))\in M_r^+(\C)\times M_t^+(\C)$, for $X\in B^+(S), Y\in B^+(S)$ such that $X_{j,j}=Y_{j,j}$, for $j=1,\dots,m$.
		\item[(4) ]  $W(\{P_S E_j P_S\}_{j=1}^m\cap W(\{P_V E_j P_V\}_{j=1}^m\neq \{(0,\dots,0)\}$.
		\item[(5) ]
		\label{item intersecc JNR de finitas mat}  $W\left(\{U_S^{-1}(P_S E_j P_S)\}_{j=1}^m\right)\cap W\left(\{U_V^{-1}(P_V E_j P_V)\}_{j=1}^m\right)\neq \{(0,\dots,0)\}$.
		\item[(6) ] The pair of subspaces $(S,V)$ form a support (see Definition \ref{soporte def}).
		\item[(7) ] If $R\in  (\mathcal{B}_S^h\oplus \mathcal{B}_V^h)^\perp\cap K^h(H)$,  $\lambda\in\R_{>0}$ and $\|R\|\leq\lambda$ then the compact operator $\lambda (P_S-P_V)+R$ is minimal.
	\end{enumerate}
\end{theorem}
\begin{proof} The equivalence between (1) and (2) is due to \eqref{eq diags = con A,-C perp Tn,Ln finitos y diags finitas} of Corollary \ref{coro Diags iguales sii perp finitos}. The definition of $p_m(m_S)$ and of $\{(T_n,L_n)\}_{n=1}^m$ jointly with Proposition \ref{prop union t psubmalfa igual a W y W} gives (2) $\Leftrightarrow$ (3). The equivalence (3) $\Leftrightarrow$ (4)  follows from  the definition of a joint numerical range and the fact that $U$ and $U^{-1}$ preserve positive definiteness. Corollary \ref{prop JNR escrito con matrices finitas} gives  (4) $\Leftrightarrow$ (5). Definition \ref{soporte def} is (1) $\Leftrightarrow$ (6) and (1) $\Leftrightarrow$ (7) can be found in Corollary 10 of \cite{bottazzi-varela-minimal-compacts} for example.
\end{proof}
\begin{remark}
	Note that the equivalence (5) of Theorem \ref{prop equivs mS cap mV no vacio} involves joint numerical ranges of $r\times r$ and $t\times t$ matrices. This allows the application of many techniques obtained for finite dimensional matrices studied and cited in \cite{kloboukvarela-mom-jnr} to describe them.
\end{remark}


\begin{thebibliography}{XX}
	\bibitem{andruchow_larotonda} Andruchow, E., Larotonda,  G. The rectifiable distance in the unitary Fredholm group, Studia Math. 196 (2) (2010), p. 151-178. 
	\hyperlink{http://dx.doi.org/10.4064/sm196-2-4}{http://dx.doi.org/10.4064/sm196-2-4}.
	
	\bibitem{BK} Bertlmann, R., Krammer,  P. Bloch vectors for qudits 
	Journal of Physics A: Mathematical and Theoretical, (2008), vol. 41, no 23, p. 235303.
	\hyperlink{http://dx.doi.org/10.1088/1751-8113/41/23/235303}{http://dx.doi.org/10.1088/1751-8113/41/23/235303}.
	
	\bibitem{BYK} Byrd, M., Khaneja, N. Characterization of the positivity of the density matrix in terms of the coherence vector representation. Phys. Rev. A (3) 68 (2003), no. 6, p. 062322.
	\hyperlink{http://dx.doi.org/10.1103/PhysRevA.68.062322}{http://dx.doi.org/10.1103/PhysRevA.68.062322}.
	
	\bibitem{dmr1} Dur\'an,  C. E.,  Mata-Lorenzo, L. E. Recht,  L. Metric geometry in homogeneous spaces of the unitary group of a C*-algebra: Part I-minimal curves. Advances in Mathematics 184(2), (2004), 342-366. \hyperlink{https://doi.org/10.1016/S0001-8708(03)00148-8}{https://doi.org/10.1016/S0001-8708(03)00148-8}.
	
	
	\bibitem{gutkin-jonckheere-karow} Gutkin, E., Jonckheere,  E. A., Karow,  M.  Convexity of the joint numerical range: topological and differential geometric viewpoints, Linear Algebra Appl., 376 (2004), 143-171.
	\hyperlink{http://dx.doi.org/10.1016/j.laa.2003.06.011}{http://dx.doi.org/10.1016/j.laa.2003.06.011}.
	
	\bibitem{BCS} Bottazzi, T., Conde, C. Sain, D. A study of orthogonality of bounded linear operators. Banach J. Math. Anal. 14 (2020), no. 3, 1001--1018. \hyperlink{http://dx.doi.org/10.1007/s43037-019-00050-0}{http://dx.doi.org/10.1007/s43037-019-00050-0}.
	
	\bibitem{bottazzi-varela-minimal-compacts} Bottazzi, T., Varela, A. Best approximation by diagonal compact operators. Linear Algebra Appl., 439.10 (2013) 3044-3056. \hyperlink{http://dx.doi.org/10.1016/j.laa.2013.08.025}{http://dx.doi.org/10.1016/j.laa.2013.08.025}.
	
	\bibitem{cho-taka} Cho, M., Takaguchi, M. Boundary points of joint numerical ranges. Pacific Journal of Mathematics, 95(1) 27-35 1981. 
	\hyperlink{http://dx.doi.org/10.2140/pjm.1981.95.27}{http://dx.doi.org/10.2140/pjm.1981.95.27}. 

	
	\bibitem{kloboukvarela-mom-jnr} Klobouk, A.H., Varela, A. Moment of a subspace and joint numerical range. Linear and Multilinear Algebra (2022), 1-34.
	\hyperlink{http://dx.doi.org/10.1080/03081087.2022.2064967}{http://dx.doi.org/10.1080/03081087.2022.2064967}.
	%
	\bibitem{li-poon} Li, C.-K.  and Poon, Y.-T. Convexity of the joint numerical range, SIAM Journal on Matrix Analysis and Applications, 21 (2000), 668-678.
	\hyperlink{http://dx.doi.org/10.1137/S0895479898343516}{http://dx.doi.org/10.1137/S0895479898343516}.
	
	\bibitem{soportes} Mendoza, A., Recht, L., Varela, A. Supports for minimal hermitian matrices. Linear Algebra Appl. 584 (2020), 458-482. 
	\hyperlink{http://dx.doi.org/10.1016/j.laa.2019.09.018}{http://dx.doi.org/10.1016/j.laa.2019.09.018}.
	
	
	\bibitem{Mu-Tom}  Müller, V., Tomilov,  Y. Joint numerical ranges: recent advances and applications minicourse by V. Müller and Y. Tomilov. With assistance from Nikolitsa Chatzigiannakidou. Concr. Oper. 7 (2020), no. 1, 133--154. \hyperlink{https://doi.org/10.1515/conop-2020-0102}{https://doi.org/10.1515/conop-2020-0102}.
	
%
	\bibitem{BarrySimon} 
	Simon, B. Trace Ideals And Their Applications. AMS, 2nd ed. (2005). \hyperlink{http://dx.doi.org/10.1090/surv/120}{http://dx.doi.org/10.1090/surv/120}.
	
	
	\bibitem{Zhang-Jiang-2022} Zhang, Y., Jiang, L. Minimal hermitian compact operators related to a C*-subalgebra of K(H). Journal of Mathematical Analysis and Applications 506.2 (2022): 125649. 
	\hyperlink{http://dx.doi.org/10.1016/j.jmaa.2021.125649}{http://dx.doi.org/10.1016/j.jmaa.2021.125649}.
	
	\bibitem{Zhang-Jiang-2023} Zhang, Y., Jiang, L. Minimal elements related to a conditional expectation in a C*-algebra. Ann. Funct. Anal. 14, 28 (2023).
	\hyperlink{http://dx.doi.org/10.1007/s43034-023-00252-6}{http://dx.doi.org/10.1007/s43034-023-00252-6}.

	
	
\end{thebibliography}
\end{document}